\documentclass{amsart}

\usepackage{amssymb,amsfonts,amsxtra,amsmath,tikz-cd,xcolor,mathrsfs,verbatim,centernot,enumitem,mathtools}
\usepackage{quiver}
\usepackage{relsize}    
\usepackage[all]{xy}
\usepackage{dsfont}
\usetikzlibrary{decorations.markings}
\tikzset{negated/.style={
        decoration={markings,
            mark= at position 0.5 with {
                \node[transform shape] (tempnode) {$\backslash$};
            }
        },
        postaction={decorate}
    }
} 
\def\@firstoftwo@second#1#2{%
  \def\temp##1.##2\@nil{##2}%
   \temp#1\@nil}
\newcommand\sref[1]{%
   \expandafter\@setref\csname r@#1\endcsname\@firstoftwo@second{#1}%
}

\newcommand{\ind}{\mathds{1}} 

\newtheorem{theo}{Theorem}[section]
\newtheorem*{theo*}{Theorem} 
\newtheorem{lemm}[theo]{Lemma}

\newtheorem{conj}[theo]{Conjecture}
\newtheorem*{conj*}{Conjecture} 
\newtheorem{coro}[theo]{Corollary}

\newtheorem{ques}[theo]{Question}

\newtheorem{prop}[theo]{Proposition}

\theoremstyle{definition} 

\newtheorem{exam}[theo]{Example}
\newtheorem*{exam*}{Example} 
\newtheorem{defi}[theo]{Definition}

\theoremstyle{remark}
\newtheorem{rema}[theo]{Remark}

\numberwithin{equation}{section}

\newcommand{\CC}{\mathbb{C}}

\newcommand{\cB}{\mathcal{B}}

\newcommand{\cK}{\mathcal{K}}
\newcommand{\cL}{\mathcal{L}}

\newcommand{\cT}{\mathcal{T}}

\newcommand{\cP}{\mathcal{P}}

\newcommand{\cJ}{\mathcal{J}}
\newcommand{\cF}{\mathcal{F}}

\def\sD{\mathscr{D}}

\def\sO{\mathscr{O}}

\def\sS{\mathscr{S}}

\def\and{\quad\mathrm{and}\quad}

\def\Hom{{\textup{Hom}}}

\def\Eu{\textup{Eu}}
\def\Der{{\textup{Der}}}
\def\PH{{\textup{PH}}} 
\def\Sch{{\textup{Sch}}}
\def\GSV{{\textup{GSV}}}

\def\Ind{{\textup{Ind}}}
\def\Sing{{\textup{Sing}}}

\def\beq{\begin{equation}}
\def\eeq{\end{equation}}

\begin{document}
\title[Chern Classes of  Foliations]{Microlocal indices and Chern Classes of Foliations}

\author{Xia Liao}
\address{Department of Mathematical Sciences, 
Huaqiao University, 
Chenghua North Road 269, 
Quanzhou, Fujian, China}
\email{xliao@hqu.edu.cn}
\thanks{The first author is supported by Chinese National Science Foundation of China (Grant No.11901214).}

\author{Xiping Zhang}
\address{School of Mathematical Sciences, Key Laboratory of Intelligent Computing and Applications (Ministry of Education),
Tongji University, 
1239 Siping Road, 
Shanghai, China}
\email{xzhmath@gmail.com}
\thanks{The second author  is supported by National Natural Science Foundation of China (Grant No.12201463). }





\begin{abstract} 
In this paper, we study how global index formulas arise in the theory of one-dimensional holomorphic foliation from the microlocal point of view. We give short proofs and generalizations to a few exisiting index formulas concerning Schwartz, GSV and logarithmic indices.

\end{abstract}
\maketitle

\section{introduction}

Let $M$ be a compact complex analytic manifold and let $L$ be a holomorphic line bundle. A one-dimensional holomorphic foliation $\cF$ on $M$ with tangent bundle $L$ is a holomorphic section $s_\cF \in H^0(M,TM\otimes L^\vee)$, and it is locally represented by holomphic vector fields via local trivilizations of $L$. Suppose $\cF$ has only isolated singularities, i.e. $s_\cF$ only vanishes at finitely many points, then the classical Baum-Bott formula states that
\begin{equation}
\label{eq0}
\int_M c(TM-L)\cap [M] = \sum_{x\in \Sing(\cF)} \Ind_{\PH}(\cF, x) 
\end{equation}
where $\Ind_{\PH}(\cF, x)$ is the Poincar\'e-Hopf index of a local vector field around $x$ representing $\cF$. When the line bundle $L$ is trivial, equation \eqref{eq0} is reduced to the classical Poincar\'e-Hopf index theorem.

In many generalizations of the Baum-Bott formula, one often considers a hypersurface $D$ and assumes $\cF$ is tangent to $D_{sm}$. In this situation there are other indices to consider, namely, the GSV index, the Schwartz index, the homological index, the logarithmic index, and if $\dim M=2$ the Camacho-Sad index and the variational index. The definitions of these indices each requires some additional assumptions on $D$ and $\cF$, and the reader may see \S\ref{sec; microlocalintersectionformula} and the references therein for more details on the GSV, Schwartz and logarithmic indices. Ingenious summations of these indices give rise to various global index formulas. Here we only mention two of them which are most relevant to this paper. 
\begin{enumerate}
\item When $D$ is normal crossing, Corr\^{e}a  and Machado in \cite{CM19} and \cite{CM24} proved that 
\[
\int_M c(TM(-\log D)-L)\cap [M]
\]
is a summation of  Poincar\'e-Hopf indices and logarithmic indices.
\item When $D$ has only isolated singularities, Machado in \cite[Theorem A]{Macha25} and  \cite[Theorem 1]{Macha25-Mar} proved that
\[
\int_M c(TM-\sO(D)-L)\cap [M]
\]
is a summation of Poincar\'e-Hopf indices and logarithmic indices minus a summation of GSV indices.
\end{enumerate} 

Proofs of these index formulas in the original papers depend on heavy computations which are riddling at some parts. To clarify what actually happens, we give new characterizations of the Schwartz, GSV and logarithmic indices by microlocal intersection formulas in \S\ref{sec; microlocalintersectionformula}. In the literature the Schwartz and GSV indices were defined by either obstruction theory or Chern-Weil theory, and the logarithmic index was defined by homological algebra. Defining them by different mathematical theories complicates the comparison of these indices, whereas the microlocal approach is capable of giving a unified framework for comparing them.

To elaborate our perspective, consider a vector field $\tilde{\nu}$ on an open set $W\subset \CC^n$ and a reduced hypersurface $D=V(f)\subset W$. We assume that $\tilde{\nu}$  is tangent to $D$ at its smooth points and has only isolated singularities.
Any Hermitian metric on $W$ induces an $\mathbb{R}$-linear $C^\infty$ isomorphism $TW \cong T^*W$ under which $\tilde{\nu}$ corresponds to a continuous section $\tilde{\sigma}\colon W\to T^*W$.
We then consider constructible functions $\ind_D$, $\Psi_f(\ind_W)$, $\ind_{W\setminus D}$ and their characteristic cycles $\textup{CC}( \ind_D)$, $\textup{CC}( \Psi_f(\ind_W))$, $\textup{CC}( \ind_{W\setminus D})$  
(see \S\ref{sec; charcycle} for  details about characteristic cycles). Here $\Psi_f(\ind_W)$ is the constructible function of nearby cycles, whose value at $x\in D$  is given by the Euler characteristic of a local
Milnor fiber of $f$ at $x$. 
In Proposition~\ref{prop; microSchFormula}, Proposition~\ref{prop; microGSVFormula}, 
Proposition~\ref{prop; micrologFormulaiso} 
and Proposition~\ref{prop; micrologFormulafree} 
we prove the following result.
\begin{theo} 
\label{t1}
Let $\Sch(\tilde{\nu}, D, 0)$, $\GSV(\tilde{\nu}, D, 0)$ and $\Ind_{\log}(\tilde{\nu}, D, 0)$ be the Schwartz index, the GSV index and the logarithmic index of $\tilde{\nu}$ at $0\in W$ respectively.   
\begin{enumerate}
\item If $D$ has an isolated singularity at $0$, denoting by $\tau_D(0)$ and $\mu_D(0)$ the Tjurina number and the Milnor number of $D$ at $0$ we have 
\begin{align*}
\Sch(\tilde{\nu}, D, 0)  = & \sharp_0 \left([\textup{CC}( \ind_D)]\cdot [\tilde{\sigma}(W)]\right)  \/,\\
\Ind_{\log}(\tilde{\nu}, D, 0) =& \sharp_0 \left([\textup{CC}(\ind_{W\setminus D})]\cdot [\tilde{\sigma}(W)]\right)+ (-1)^{n-1} (\tau_D(0)-\mu_D(0))\/.	
\end{align*}
\item If the vector field $\tilde{\nu}$  satisfies  ($\ast\ast$) (see \S\ref{sec; microlocalGSV} for details),  we have 
\[
\GSV(\tilde{\nu}, D, 0)   =  \sharp_0 \left([\textup{CC}(\Psi_f(\ind_W)]\cdot [\tilde{\sigma}(W)]\right)  \/.
\]

\item If $D$ is  holonomic, strongly Euler homogeneous and free  at $0$ (see \S\ref{sec; logder} for details),  we have
\[
\Ind_{\log}(\tilde{\nu}, D, 0)=  \sharp_0 \left([\textup{CC}(\ind_{W\setminus D})]\cdot [\tilde{\sigma}(W)]\right)  \/.
\]
\end{enumerate}
\end{theo} 

\begin{rema}
The microlocal approach to local and global indices has a long history. It was initiated by Kashiwara in \cite{MR368085} for holonomic D-modules, and was further developed in \cite[Chapter 8,9]{MR1299726}, \cite{MR2059229}, \cite{MR2031639} for constructible sheaves and constructible functions. In a series of papers (see for example \cite{MR2025268},\cite{MR2171307}), Ebeling and Gusein-Zade introduced local Euler obstructions for 1-forms, but they never related their indices to microlocal intersection numbers. In private communication after the completion of this paper, Sch\"{u}rmann informed us that about 20 years ago he found a way to express local indices from the theory of vector fields by microlocal intersection numbers which is very similar to our Theorem \ref{t1}, but the result was never published.  
\end{rema}

Based on this theorem, any relation among constructible functions $\ind_D, \Psi_f(\ind_W), \ind_{W\setminus D}$ gives rise to a relation among $\Sch(\tilde{\nu}, D, 0),\GSV(\tilde{\nu}, D, 0),\Ind_{\log}(\tilde{\nu}, D, 0)$ (see Remark \ref{rema:GSV-SCH=mu} for a simple example). This theorem also demonstrates the scope and limitation of our approach. For any interesting constructible function $\gamma$ one can define an index by $ \sharp_0 \left([\textup{CC}(\gamma)]\cdot [\tilde{\sigma}(W)]\right)$ provided the local intersection number is well-defined, i.e. the covector $\tilde{\sigma}(0)$ is an isolated point in the set $\tilde{\sigma}(W) \cap |\textup{CC}(\gamma)|$. For example, the Poincar\'{e}-Hopf index admits such an interpretation for $\gamma= 1_W$ (see Example \ref{exam; microPH}). On the other hand, since the Camacho-Sad index can be any complex number, it cannot be obtained as an intersection number in the above manner. 

Since the local Euler obstructions form a  basis for the group of constructible functions, it is natural to ask when is 
\[
 \sharp_0 \left([\textup{CC}(\Eu_X)]\cdot [\tilde{\sigma}(W)]\right)= (-1)^d\cdot\sharp_0 \left([T^*_XW]\cdot [\tilde{\sigma}(W)]\right)
\] 
well-defined, where $X$ is an irreducible analytic subvariety of dimension $d$ containing $0$. We show in \S\ref{sec; Eulerobstruction} that it is well-defined when $\tilde{\nu}$ is tangent to $X$ at its smooth points and has an isolated singularity at $0$. Using a duality between the Nash blowup of $X$ and the conormal space $T^*_XW$, we can transform the intersection $[T^*_XW]\cdot [\tilde{\sigma}(W)]$ into an intersection in the Nash tangent bundle. This leads to a  Gonz\'{a}lez-Sprinberg type formula for $\sharp_0 \left([\textup{CC}(\Eu_Z)]\cdot [\tilde{\sigma}(W)]\right)$ (see Proposition \ref{prop; EulerObs} and Theorem \ref{theo:microlocal for Euobs}). In particular, we obtain in Corollary \ref{coro:algebraicGSV} an algebraic formula for the GSV index at an isolated singularity of $D$ that does not appeared to have been recorded in the literature before.

Returning to the discussion of global index formulas, it is now an easy matter to manufacture them in large quantities. Indeed, fixing any Hermitian metric on $TM\otimes L^\vee$ induces a $\mathbb{R}$-linear $C^\infty$ isomorphism $TM\otimes L^\vee \cong T^*M \otimes L$. The holomorphic section $s_\cF$ defining the foliation is tranformed to a section $\sigma_\cF$ of $T^*M\otimes L$. Given any constructible function $\gamma$, we can intersect $\sigma_\cF(M)$ with the jet characteristic cycle $\textup{JCC}(\gamma)$ which is just $\textup{CC}(\gamma)$ twisted by $L$ and obtain a summation of a bunch of indices, provided that $\sigma_{\cF}(M)$ meets $\textup{JCC}(\gamma)$ only at finitely many points. Deforming $\sigma_\cF$ to the zero section of $T^*M\otimes L$ does not change the total index, and Aluffi's twisted shadow formula Proposition \ref{prop; jettwistformula} helps us to express the intersection with the zero section as an integral. 

Applying this method to the constructible function $\gamma=\ind_{M\setminus D}$, and assuming either D is strongly Euler homogeneous, holonomic and free, or assuming D has only isolated singularities and $L$ is very ample, we obtain a generalization of Corr\^{e}a and Machado's theorem (\cite[Theorem 2]{CM24}) about $\int_M c(TM(-\log D)-L)\cap [M]$ in Corollary \ref{coro; SEH}.    
Applying it to $\gamma=\Psi_f(\ind_M)$ without hypothesizing singularities of $D$, we obtain a generalization of 
Suwa's theorem \cite[Theorem 7.16]{Suwa} about summation of $\GSV$ indices on $D$ in Corollary \ref{coro; sumGSVonD} and a generalization of 
Machado's theorem (\cite[Theorem A]{Macha25}) about $\int_M c(TM-\sO(D)-L)\cap [M]$ in Corollary \ref{coro; iso}.

Here we leave the burden to the reader to find the precise statements and formulas in the corollaries  to avoid congesting the introduction by big formulas.

In general one may expect  to write down   global index formulas for other interesting constructible functions   such as those associated to the intersection cohomology sheaves.  
From our explanation, the  difficulty   lies primarily at writing down explicitly what the characteristic cycle is and finding the right conditions for the foliation to make sure all the local intersections are well-defined.

\hspace{1cm}

\noindent{\bf{Acknowledgement}}:
We thank Arturo Fern\'{a}ndez-P\'{e}rez and J\"{o}rg Sch\"{u}rmann for carefully reading the earlier drafts of the paper and giving very helpful comments.

\section{Preliminary}
\label{sec; preliminary}

\subsection{Logarithmic Derivations and Free Divisors}
\label{sec; logder}
Let $M$ be a complex manifold of dimension $n$ and $D$ be a reduced hypersurface on $M$. 
\begin{defi}
\label{defn; logDer}
The sheaf of logarithmic derivations along $D$, denoted by $\textup{Der}_M(-\log D)$,  is locally given by 
\begin{equation*}
\Der_M(-\log D)(U)=\{\chi \in \Der_X(U) \ | \  \chi(h) \in h\mathscr{O}_U\}
\end{equation*}
for any open subset $U\subset M$ and any local defining equation  $h$ of $D$ on $U$.   
\end{defi} 
This sheaf was  introduced  by Saito  in \cite{MR586450},  where he proved that it is a coherent and reflexive $\sO_M$-module, and has generic  rank $n$ outside $D$.  

\begin{defi}
We say $D$ is free at $p\in D$ if $\Der_{M,p}(-\log D)$ is a free $\sO_{M,p}$-module. We
say $D$ is a free divisor if   the coherent sheaf $\Der_M(-\log D)$ is locally free. 
\end{defi}
 
Given a free divisor $D$, the logarithmic tangent bundle $TM(-\log D)$ is the vector bundle whose sheaf of sections is $\Der_M(-\log D)$. The logarithmic cotangent bundle $T^*M(\log D)$ is the dual of $TM(-\log D)$, whose sheaf of sections is $\Omega^1_M(\log D)$. The natural inclusion $\Der_M(-\log D) \hookrightarrow \Der_M$ induces a linear map between rank $n$ vector bundles $j\colon T^*M \to T^*M(\log D)$. We  call it the logarithmic map of the pair $(M,D)$.

\begin{defi}[\cite{MR3366865}]
\label{defn; EH}
The divisor $D$ is strongly Euler homogeneous at a point $p\in D$ if there exists a local equation $h\in \sO_{M,p}$ of $D$ and a local Euler vector field $\chi\in \mathfrak{m}_{M,p}\cdot \Der_{M,p}$  such that $\chi(h)=h$.   
The divisor $D$ is strongly Euler homogeneous if it is so at any point $p\in D$.
\end{defi} 

Following  \cite{MR586450},  $D$ admits a unique   stratification   such that the tangent space of each stratum consists of the evaluations of the logarithmic vector fields along $D$.   Such a stratification is called the logarithmic stratification of $D$. 
\begin{defi} 
\label{defn; holonomic}
$D$ is called  holonomic  if the logarithmic stratification  is locally finite.
\end{defi} 

\begin{exam}
Normal crossing divisors are holonomic, strongly Euler homogeneous and free.
\end{exam} 

\begin{exam}
Any reduced divisor on a complex smooth surface is always holonomic and free. It is strongly Euler homogeneous if and only if it is quasi-homogeneous.
\end{exam}

One of the key features of holonomic  strongly Euler homogeneous  free divisors  is  that the ambient space $M$ is log transverse to $D$ (see \cite[Definition 4.1]{LZ24-1}). We will come back to this feature in Proposition~\ref{prop; hSEHchernformula}.

\subsection{Characteristic Cycles}
\label{sec; charcycle}
\begin{defi} 
\label{defn;conormal}
Let $Z$ be any irreducible subvariety of a complex analytic manifold $M$. 
The   conormal space of $Z$ in $M$, denoted by $T^*_ZM$, is the closure of  the following space in $T^*M$.
 \[
 T^*_{Z_{sm}}M:=\Big\{ (x, \xi) \ \vert \ x\in Z_{sm},  \ \xi\in T^*M|_x \text{ and }   \xi(T_x Z)=0 \Big\}
 \]
 \end{defi}
This is a conical Lagarangian subvariety of $T^*M$. It can also be shown that any irreducible conical Lagarangian subvariety of $T^*M$ has the form $T_Z^*M$ for some subvariety $Z$. 
Let $\cL(M)$ be the free abelian group generated by conical Lagarangian subvarieties, then we have $\cL(M)\cong \mathcal{Z}(M)$, the group of cycles of $M$.

Given any line bundle $L$ on $M$ we may also consider a $L$-twisted version of  the conormal space, i.e., the jet conormal space $JT^*_ZM$ introduced by the first author
in \cite{MR4749159}. We also refer to   \cite[Definition 3.7]{LZ24-1} for another treatment.  In this paper we only need the following characterization.
\begin{prop}
\label{prop; jetconormal}
The jet conormal space $JT^*_ZM$ is the  unique $n$-dimensional subspace of the twisted cotangent bundle $T^*M\otimes L$ such that for every open subset $U$ of $M$ that trivializes $L$ and $T^*M$, under  the trivialization map $\iota_U\colon T^*M|_U\cong T^*U\cong    T^*M\otimes L|_U$ we have 
\[
\iota_U(T_Z^*M|_U) =  JT^*_ZM|_U  \/.
\]
\end{prop}
In other words, the jet conormal space and the ordinary conormal space are locally identical; their difference arises during the gluing process of the globalization. We denote by $J\cL(M)$ the free abelian group generated by jet conormal spaces, clearly $J\cL(M)\cong \cL(M)$.

A constructible function on $M$ is an integer valued function $\gamma\colon M\to \mathbb{Z}$ such that $\gamma^{-1}(i)$ is a constructible subset of $M$. The constructible functions on $M$ form an abelian group denoted by  $CF(M)$.  
By a basis of $CF(M)$ we mean a subset $\mathcal{B}$ of $CF(M)$  such that any constructible function is uniquely expressed as a locally
finite linear combination of elements in $\mathcal{B}$ with $\mathbb{Z}$-coefficient. 
Besides the natural  basis of indicator functions  $\{\ind_Z\}$ for all subvarieties $Z$, 
in  \cite{MAC} Macpherson introduced another  basis  as the key ingredient of his generalization of the Poincar\'e-Hopf theorem to singular varieties. 
This is the basis of local Euler obstruction functions $\{\textup{Eu}_Z\}$, defined for all closed subvarieties $Z$ of $M$.     
To simplify notation we will denote the signed local Euler obstruction function $(-1)^{\dim Z}\textup{Eu}_Z$ by $\textup{Eu}^\vee_Z$.
\begin{defi}
\label{defn; JCC}
Let $\gamma=\sum_\alpha m_\alpha \textup{Eu}^\vee_{Z_{\alpha}}$ be a constructible function on $M$, i.e.,  a (locally) finite linear combination with $m_\alpha\neq 0$ and $Z_\alpha\neq Z_\beta$ if  $\alpha\neq \beta$. 
\begin{itemize}
\item[$\bullet$]  The  (ordinary) characteristic cycle  of $\gamma$ is the (locally) finite linear combination 
\[
\textup{CC}(\gamma):= \sum_{\alpha}  m_\alpha \cdot  T^*_{Z_\alpha}M  \in \cL(M)\/.
\]
\item[$\bullet$]  The  jet characteristic cycle  of $\gamma$ is  the (locally) finite linear combination 
\[
\textup{JCC}(\gamma):= \sum_{\alpha}   m_\alpha \cdot  JT^*_{Z_\alpha}M  \in J\cL(M)\/.
\]
\end{itemize}  
\end{defi}
We recommend \cite[\S 10.3.3-10.3.4]{MS22} for a brief introduction to the theory of characteristic cycles in the complex analytic context  and \cite[Chapter 5.2]{MR2031639} for a thorough treatment of characteristic cycles in the real analytic context via stratified Morse theory.

Characteristic cycles are closely related to the following celebrated theorem of Macpherson.  
\begin{theo}[\cite{MAC}] 
There exists a unique natural transformation $c_*$ from the functor $CF(\bullet)$ to the even-dimensional Borel-Moore homology $H_{2*}(\bullet)$ such that  
\[
c_*(\ind_M)=c(TM)\cap [M] \text{ whenever } M \text{ is smooth } \/.
\]
\end{theo}

Given a constructible function $\gamma$ on $M$, the standard way to obtain $c_*(\gamma)$ is to apply a so-called shadow operation to $\textup{CC}(\gamma)$. 
The name shadow was coined by Aluffi in \cite{MR2097164} (see also \cite{AMSS23}) but the actual operation first appeared in much earlier papers \cite{MR804052} and \cite{MR1063344}. Let $E$ be a vector bundle of rank $e$ on $M$ and $C$ be a $\CC^*$-invariant cycle of $E$. By \cite[Proposition 2.7]{AMSS23}, the shadow of $C$ cast  in $E$ is characterized by 
\[
\textup{Shadow}_{E}([C])= \left( [C]_{\CC^*} \cdot [M]_{\CC^*} \right)|_{t=1} \/.
\]
Here $M$ denotes the zero section of $E$ and $[\bullet]_{\CC^*}$ denotes the $\CC^*$-equivariant fundamental cycles. The $\CC^*$-equivariant intersection product lives in $H^{\CC^*}_{2*}(E)\cong H_{2*}(M)[t]$, the polynomial ring  with coefficients in $H_{2*}(M)$. Then the shadow is the homology class in $H_{2*}(M)$  by specializing to $t=1$. 

There is a simple minded duality operation on $H_{2*}(M)$ which sends a (not necessarily homogeneous) homology class $\alpha = \sum_i \alpha_i$ (beware that $\alpha_i\in H_{2i}(M)$) to
\[
\textup{Dual}(\alpha)=\sum_i (-1)^i\alpha_i.
\] 
In other words, the sign of the $2i$-dimensional component is modified by $(-1)^i$. We use the notation
\[
\textup{Dual-Shadow}_{E}([C])
\]
for the dual of the shadow of $C$.

\begin{prop}[{\cite[Lemma 4.3]{MR2097164}}]  
\label{prop; shadow}
We have the following relation in $H_*(M)$.
\[
c_*(\gamma) = \textup{Dual-Shadow}_{T^*M}\left([\textup{CC}(\gamma)] \right) \/.
\]
In particular  when $M$ is compact, denoting by $M$ the zero section of $T^*M$ we have 
\[
 \int_M c_*(\gamma) =  \int_M \textup{Shadow}_{T^*M}\left([\textup{CC}(\gamma)] \right)
=\deg \left([\textup{CC}(\gamma)]\cdot [M]  \right) \/.
\]
\end{prop}

For holonomic  strongly Euler homogeneous  free divisors  we have  the following Chern class formula from 
\cite[Corollary 4.7 and Corollary 5.10]{LZ24-1}.
\begin{prop}
\label{prop; hSEHchernformula}
Let $D$ be a  holonomic strongly Euler homogeneous  free divisor. Then $M$ is log transverse to $D$ and we have
\[
 c_*(\ind_{M\setminus D}) =\textup{Dual-Shadow}_{T^*M(\log D)}((-1)^n[M])=c(TM(-\log D)) \cap [M] \/.
\]
\end{prop}

Recall that the jet characteristic cycle and the ordinary characteristic cycle are locally identical but globally different. The difference is captured by the following twisted shadow formula from \cite[Lemma 2.2]{AMSS23}.
\begin{prop}
\label{prop; jettwistformula}
Let $\textup{Shadow}(\textup{CC}(\gamma))=\sum_j\alpha_j$. We have
\[
\int_M \textup{Shadow}_{T^*M\otimes L} \left([\textup{JCC}(\gamma)] \right)=\deg \left([\textup{JCC}(\gamma)]\cdot [M]  \right) = \sum^{n}_{j=0}\int_M c_1(L)^{j} \cap \alpha_j \/.
\]
where $M$ is regarded as the zero section of $T^*M\otimes L$.
\end{prop}

\subsection{Vector Fields}

A vector field $v$ on an oriented $n$-dimensional $C^\infty$ manifold $N$ is a continuous section of $TN$. The singularity locus of $v$, denoted by $\Sing(v)$, is the set of zeros of $v$. 
\[
\Sing(v):=\{x\in N \ \vert \ v(x)=0\}. 
\]

\begin{defi}
The Poincar\'e-Hopf Index $\Ind_{\PH}(v, x)$ of the vector field $v$ at an isolated singularity $x\in \Sing(v)$ is the degree of the map $\frac{v}{\left\lVert v\right\rVert}: S^{n-1}_\epsilon \to S^{n-1}_1$ from a sufficiently smalll sphere of radius $\epsilon$ around $x$ in $N$ to the unit sphere in $\mathbb{R}^m$. Here the orientation of $S^{n-1}_\epsilon$ is induced from the orientation of the ambient manifold $N$.
\end{defi}

\begin{defi}
Let $X$ be an anlytic subvariety of a complex manifold $M$. A stratified vector field $v$ on $X$ is a continuous section of $TM\vert_X$ together with a Whitney stratification $\sS=\{\sS_\alpha\}$ of $X$ such that $v$ is tangent to $\sS_\alpha$ for every $\alpha$.
\end{defi}

A holomorphic vector field on a complex manifold $M$ is a holomorphic section of $TM$. Fix a holomorphic cordinate system $x_1,\ldots,x_n$ aroung a point $x\in M$, a holomorphic vector field $v$ can be written locally as $v=\sum_{i=1}^n a_i\partial_{x_i}$ with $a_i \in \sO_{\CC^n,0}$. We denote the ideal $(a_1,\ldots,a_n)$ of $\sO_{\CC^n,0}$  by $J_x(v)$, it  has finite colength if and only if $v$ has isolated singularity at $x$.

\begin{defi}
The multiplicity $\mu_x(v)$ of the holomorphic vector field $v$ at an isolated singularity $x$ is the number $\dim_{\CC} \frac{\sO_{\CC^n, 0}}{J_x(v)}$.
\end{defi} 

It is well known that $\Ind_{\PH}(v, x)=\mu_x(v)$ if $x$ is an isolated singularity of a holomorphic vector field $v$ (see for example \cite[p.86 Theorem 1]{MR777682}).

\subsection{One-dimensional  Singular Foliations}
Let $M$ be a  complex manifold and $TM$ be the tangent bundle of $M$. Let $L$ be a line bundle on $M$.
\begin{defi} 
A one-dimensional (singular) holomorphic foliation  $\cF$ on $M$  with tangent bundle $L$  is
given by the following data:
\begin{enumerate}
\item an open covering $\{U_\alpha|\alpha\in I\}$ of $M$;
\item a holomorphic vector field $\nu_\alpha\in Der_M(U_\alpha)$ for each $\alpha\in I$;
\item for every  intersection $U_\alpha\cap U_\beta\neq \emptyset$, a holomorphic function 
$g_{\alpha\beta}\in \sO_M^*(U_\alpha\cap U_\beta)$ such that 
\[
\nu_\alpha|_{U_\alpha\cap U_\beta} =g_{\alpha\beta}\cdot \nu_\beta|_{U_\alpha\cap U_\beta} \text{ and }  \{g_{\alpha\beta}|\alpha, \beta\in I\}\in H^1(M, \sO_M^*) \text{ is a cocycle of }   L^\vee \/.
\]
\end{enumerate}
Equivalently, $\cF$ is given by  a 
global holomorphic section in $H^0(M, TM\otimes L^\vee)=\Hom(L, TM)$. We will denote  this   global holomorphic section by $s_\cF$.  
\end{defi}

\begin{defi}
The singular set of $\cF$ is defined by 
\[
\Sing(\cF):=\{s_\cF=0\} =\cup_{\alpha\in I}  \Sing(\nu_\alpha) \/.
\]
\end{defi}

\begin{defi}
Given a reduced divisor $D\subset M$, a logarithmic one-dimensional foliation along $D$ (with tangent bundle $L$) is 
 a one-dimensional foliation $\cF$ such that $s_\cF$ factors as 
 \[
 s_\cF\colon M \longrightarrow TM(-\log D)\otimes L^\vee \longrightarrow TM\otimes L^\vee  \/,
 \] 
 where the last morphism is induced by $\Der_M(-\log D)\hookrightarrow \Der_M$.

 Equivalently this says that 
$\nu_\alpha\in \Der_{U_\alpha}(-\log D\cap U_\alpha)$ are logarithmic vector fields for each $\alpha\in I$. 
\end{defi}

\section{Euler obstructions for vector fields} 
\label{sec; Eulerobstruction}

Let $W$ be an open subset of $\CC^n$ and let $X$ be a purely $d$-dimensional analytic subvariety of $W$.
Let $\pi\colon Z\to X$ be the   Nash blowup of $X$ and $\cT$ be the Nash tangent bundle. Let  $\tilde{\nu}$ be a continuous vector field on $W$ tangent to $X$ at smooth points of $X$ and let $\nu=\tilde{\nu}|_X$ be its restriction on $X$. We always assume $0\in X$ and $\Sing(\nu)=\{0\}$.

Such $\nu$ naturally lifts to a continuous section $\chi$ of the Nash tangent bundle $\cT$. 
Because $\Sing(\nu)=0$, $\chi$  only vanishes  at $Z_0:=\pi^{-1}(0)$. Take a ball $B_\epsilon$ of radius $\epsilon$  small enough such that 
$S_\epsilon=\partial B_\epsilon$ intersects each stratum  $\sS_\alpha$ transversely. We define $V_\epsilon:=\pi^{-1}(X\cap B_\epsilon)$ and $\partial V_\epsilon:=\pi^{-1}(X\cap S_\epsilon)$.

\begin{defi} 
The local Euler obstruction $\textup{Eu}_X^{\tilde{\nu}}(0)$ of $\tilde{\nu}$ at $0$ is defined to be the obstruction to extending $\chi$ as a nowhere zero section of $\cT$ from $\partial V_\epsilon$ to $V_\epsilon$. Equivalently,  
\[
\textup{Eu}_X^{\tilde{\nu}}(0)= \int \chi^*(\tau)\cap [Z]. 
\]
where $\tau$ is the Thom class of $\cT$. Note that our assumptions on $\nu$ implies that $\chi^*(\tau)\in H^{2d}(V_\epsilon, \partial V_\epsilon,\mathbb{Z})$ so that $\textup{Eu}_X^{\tilde{\nu}}(0)$ is well-defined. Clearly $\textup{Eu}_X^{\tilde{\nu}}(0)$ depends only on $\nu$. 
\end{defi}

\begin{prop}
\label{prop; EulerObs}
We have the following Gonz\'{a}lez-Sprinberg type formula when $\tilde{\nu}$ is holomorphic.
\[
\textup{Eu}_{X}^{\tilde{\nu}}(0)=\int_{Z_0}  [Z]\cdot [\chi(Z)]=\int_{Z_0} c(\cT)\cap s(\tilde{Z}_0, Z)  \/.
\]
Here $\tilde{Z}_0$ is the intersection scheme of $\chi(Z)$ with the zero section of $\cT$ and $\int_{Z_0}: H_0(Z_0) \to \mathbb{Z}$ is the pushdown morphism induced by $Z_0\to pt$.
\end{prop} 

\begin{proof}
The section $\chi$ is holomorphic when $\tilde{\nu}$ is holomorphic. Since the Thom class of $\cT$ is represented by $[Z]$, $\textup{Eu}_{X}^{\tilde{\nu}}(0)=\int [Z]\cdot [\chi(Z)]$ follows from the definition of $\textup{Eu}_{X}^{\tilde{\nu}}(0)$. The other expression of $\textup{Eu}_{X}^{\tilde{\nu}}(0)$ follows from \cite[Proposition 6.1(a)]{MR1644323} and \cite[\S 19.2]{MR1644323}.
\end{proof}

Any Hermitian metric on $\CC^n$ induces an $\mathbb{R}$-linear $C^\infty$ isomorphism $T\CC^n \cong T^*\CC^n$. Under this isomorphism $\tilde{\nu}$ corresponds to a continuous section $\tilde{\sigma}\colon W\to T^*W$.

\begin{prop}
$\tilde{\sigma}(W)$ only intersects $T_X^*W$ at $(0,0)$.
\end{prop}

\begin{proof}
Denote a chosen Hermitian metric on $\CC^n$ by $h( \cdot,\cdot )$. Then $\tilde{\sigma}(x)=h_x ( \cdot, \tilde{\nu}(x))$ for any $x\in X$.

By the definition of conormal spaces, $(x, \tilde{\sigma}(x) )\in T_X^*W$ if $\tilde{\sigma}(x)$ annihilates some limiting tangent space $T$ at $x\in X$.  Since $\tilde{\nu}(x) \in T$ for any limiting tangent space $T$, 
\[
0=\tilde{\sigma}(x)\left(\tilde{\nu}(x)\right)=h_x(\tilde{\nu}(x),\tilde{\nu}(x)).
\] 
Therefore $\tilde{\nu}(x)=\nu(x)=0$, and hence $x=0$ because $\Sing(\nu)=\{0\}$.
\end{proof} 
 
\begin{theo}\label{theo:microlocal for Euobs}  
We have the following microlocal intersection formula
\[
\textup{Eu}_{X}^{\tilde{\nu}}(0)=(-1)^{d}  \sharp_0 \left( [T_X^*W]\cdot [\tilde{\sigma}(W)]  \right) 
\]
where $\sharp_0$ means the intersection number at $(0,0) \in T^*W$.
\end{theo} 

The intersection product in the theorem is defined by cup product  (see \cite[9.2.16]{MR1299726} and \cite[p.286]{MR2031639})
\[
H^{2n}_{T^*_XW}(T^*W,\mathbb{Z}) \otimes H^{2n}_{\tilde{\sigma}(W)}(T^*W,\mathbb{Z}) \xrightarrow{\cup} H^{4n}_{(0,0)}(T^*W,\mathbb{Z})\cong \mathbb{Z}.
\]
We caution the reader that in the identification $ H^{4n}_{(0,0)}(T^*W,\mathbb{Z})\cong \mathbb{Z}$, loc.cit used a real orientation of $T^*W$  which is $(-1)^n$ times the complex orientation of $T^*W$. Also, the orientation of characteristic cycles used in loc.cit is a real orientation. More precisely, $[T^*_XW]_r=(-1)^{n-d} \cdot [T^*_XW]_c$ (see \cite[p.285]{MR2031639}, \cite[Remark 5.0.3]{MR2031639}). Therefore, cancelling out the $(-1)^n$ factor we have
\[
\sharp_{0,r}\left([T^*_XW]_r \cap [\tilde{\sigma}(W)]\right) = (-1)^d \sharp_{0,c}\left([T^*_XW]_c \cap [\tilde{\sigma}(W)]\right)=\textup{Eu}_{X}^{\tilde{\nu}}(0).
\]
This explains the appearance of the sign $(-1)^d$.

\begin{proof}
Consider the short exact sequence on the Nash blowup $Z$:
\[
\begin{tikzcd}
0 \arrow[r] & Q^\vee \arrow[r] &  \pi^*T^*W  \arrow[r, "r"] & \cT^\vee \arrow[r] & 0
\end{tikzcd}
\] 
Note that by definition $T^*_XW$ is the proper image of $Q^\vee$ under the projection $p:\pi^*T^*W \to T^*W$.

The section $\tilde{\sigma}$ is pulled back to a section $\tilde{\rho}\colon Z\to \pi^*T^*W$, which  induces a section  
  $\rho  \colon Z\to \cT^\vee$ by composing with the projection $\pi^*T^*W \to \cT^\vee$.  
 We   have the commutative diagram below.
\[
\begin{tikzcd} 
Z_0 \arrow[d, " "] \arrow[r,  ""] & Z \arrow[d, "\tilde{\rho}"'] \arrow[dr, "\rho"] \\
  Q^\vee \arrow[d, "p'"] \arrow[r, hook, ""] &  \pi^*T^*W \arrow[d,  "p"] \arrow[r,   "r"] 
  &  \cT^\vee    \\
  T_X^*W \arrow[r, hook, ""] & T^*W    &   
 \end{tikzcd}\/.
\]
The morphism $r$ restricts to an isomorphism on $\tilde{\rho}(Z)$. 
Regard $Z$ as the zero section of $\cT^\vee$, we have 
\[r_*\left([\tilde{\rho}(Z)]\cdot [Q^\vee] \right)
 = r_*\left([\tilde{\rho}(Z)]\cdot r^*[Z] \right)  
 = r_*[\tilde{\rho}(Z)]\cdot  [Z]
 = [\rho(Z)]\cdot  [Z] \/.
\]
The Hermitian metric induces an $\mathbb{R}$-linear $C^\infty$ isomorphism $\cT \cong \cT^\vee$ under which $\chi$ corresponds to $\rho$. The isomorphism $\cT \cong \cT^\vee$ sends $[\chi(Z)]$ to $(-1)^d[\rho(Z)]$ where both $[\chi(Z)]$ and $[\rho(Z)]$ are equipped with the orientation induced from the complex orientation of $Z$. We have
\[
\int_{Z_0} \left([Z]\cdot [\chi(Z)]\right) =(-1)^{d} \int_{Z_0}  \left( [Z]\cdot [\rho(Z)] \right)
 =(-1)^{d} \int_{Z_0} \left( [\tilde{\rho}(Z)]\cdot [Q^\vee] \right) =(-1)^d \sharp_0 \left([T^*_XW]\cdot [\tilde{\sigma}(W)]\right)  \/.
\]
The last equality follows from the projection formula because $p_*[{Q^\vee}]=[T^*_XW]$ and $p^{-1}(\tilde{\sigma}(W))=\tilde{\rho}(Z)$. 
\end{proof} 
\begin{rema}
The right hand side of the  formula in Theorem~\ref{theo:microlocal for Euobs} is independent of the choice of the Hermitian metric used to define $\tilde{\sigma}$. In fact, given two Hermitian metrics $h_1$ and $h_2$ one may form a smooth family of  Hermintian metrics $th_1+(1-t)h_2$ and obtain a smooth family of section $\tilde{\sigma}_t$. 
The intersection product $\sharp_0 \left( [T_X^*W]\cdot [\tilde{\sigma}_t(W)]  \right) $ stays constant by the law of conservation of numbers.
\end{rema}

\begin{defi}
Given a  constructible function $\gamma=\sum_{\alpha} m_\alpha\cdot \Eu_{X_\alpha}$  on $X$, 
we say the vector field $\tilde{\nu}$ is tangent to $\gamma$, 
if $\nu$ is tangent to  (a Zariski-dense open part of) the smooth part of $X_\alpha$ for every $\alpha$. 
\end{defi}
For example, if $X$ admits a complex analytic Whitney stratification $\{\sS_\alpha\}$ and $\tilde{\nu}$ is tangent to all $\sS_\alpha$, then $\tilde{\nu}$ is tangent to any $\sS$-constructible function $\gamma$ (see \cite[10.2.1]{MS22} for the definition of $\sS$-constructibility).

\begin{defi} 
Suppose $\tilde{\nu}$ is tangent to $\gamma$. We define the local Euler obstruction of $\tilde{\nu}$ at $0$ weighted by $\gamma$ by 
\[
\Eu (\tilde{\nu},\gamma,0):=\sum_{\alpha} m_\alpha \cdot \Eu_{X_\alpha}^{\tilde{\nu}}(0) \/.
\] 
\end{defi}

This notation is a bit awkward since by definition we have $\Eu_X^{\tilde{\nu}}(0)=\Eu(\tilde{\nu},\Eu_X,0)$.

\begin{coro}\label{coro:intersection formula for Euobs}
If $\tilde{\nu}$ is tangent to the constructible function $\gamma$, then
\[
\Eu(\tilde{\nu},\gamma,0)=\sharp_0 \left(\textup{CC}(\gamma)\cdot [\tilde{\sigma}(W)]\right).
\]
If moreover $\tilde{\nu}$ is holomorphic, then the local intersection number above can be computed algebraically according to Proposition \ref{prop; EulerObs}.
\end{coro}

\begin{exam}\label{exam; microPH}
Suppose $\Sing(\tilde{\nu})=\{0\}$, then
\[
\Ind_{\PH}(\tilde{\nu}, 0)=\Eu(\tilde{\nu}, \ind_W,0) =\sharp_0 \left(\textup{CC}(\ind_W)\cdot [\tilde{\sigma}(W)]\right)= (-1)^n \cdot \sharp_0 \left( [W] \cdot [\tilde{\sigma}(W)] \right).
\]
\end{exam}
 
\begin{rema}
Let $\gamma$ be a constructible function on $X$. For any differentiable one-form $\omega$ that is dimensional transverse to $\textup{CC}(\gamma)$, in  \cite[Example 3.39]{MS22} the authors defined the local Poincar\'e-Hopf index of $\omega$ with respect to $\gamma$ using microlocal intersection and obtained a Poincar\'e-Hopf type formula for the Euler integral $\int_X \gamma d\chi$ using these indices. 
\end{rema}

Next, we consider the case that $X=D$ is a hypersurface in $W$ defined by a holomorphic function $f$. Let $\tilde{\nu}\in H^0(W,\Der_W(-\log D))$ be a logarithmic vector field such that $\Sing(\tilde{\nu})=\{0\}$.
Recall that $\Psi_f(\ind_W)$ (resp, $\Phi_f(\ind_W)$) is the constructible function whose value at $x\in D$ is   given by the
 Euler characteristic (resp, reduced Euler characteristic) of a local Milnor fiber of 
 $f$ at $x$, with relation $\Psi_f(\ind_W)=\Phi_f(\ind_W)+\ind_D$.

\begin{prop}\label{prop: nu is tangent to D}
The vector field $\tilde{\nu}$ is tangent to $\Psi_f(\ind_W)$ and $\ind_D$.
\end{prop}

\begin{proof}
Let $\Psi_f(\ind_W)=\sum m_\alpha \Eu_{X_\alpha}$ and $\Phi_f(\ind_W)=\sum m_\beta \Eu_{X_\beta}$. Let $x\in D$ be a point different from 0. Since $\tilde{\nu}(x)\neq 0$, by Saito's triviality lemma \cite{MR586450} we can find a holomorphic coordinate system $z_1,\ldots,z_n$ on a small neighbourhood $W_x$ of $x$ in $W$ with the properties
\begin{enumerate}
\item $D\cap W_x$ is defined by a holomorphic function $g(z_1,\ldots,z_{n-1})$ in $W_x$,
\item $\tilde{\nu}=\partial_{z_n}$.
\end{enumerate}  
In other words, we have a trivialization $(W,x)\cong (\CC^{n-1},0)\times (\CC,0)$ and $(D,x)\cong (D',0)\times (\CC,0)$ for some hypersurface germ $(D',0)$ in $(\CC^{n-1},0)$. By the fact that the nearby and vanishing cycle functors commute with smooth pullback, the Euler obstructions appearing in the decompositions of $\Psi_g(\ind_{W_x})$ and $\Phi_g(\ind_{W_x})$ all take the form $\Eu_{Y}$ with $(Y,x)\cong (Y',0)\times (\CC,0)$, so $\tilde{\nu}(x)$ is tangent to any of these varieties $Y$ whenever $Y$ is smooth at $x$. Since these varieties $Y$ are also intersections of $X_\alpha$ and $X_\beta$ with $W_x$, we conclude that $\tilde{\nu}(x)$ is tangent to $X_\alpha$ and $X_\beta$ whenever they are smooth at $x$. Finally, since $\tilde{\nu}(0)=0$, so it is a tangent vector at $0$ for any of those $X_\alpha$ and $X_\beta$ smooth at $0$. We have shown that $\tilde{\nu}$ is tangent to $\Psi_f(\ind_{W})$ and $\Phi_f(\ind_{W})$. Since $\ind_D=\Psi_f(\ind_{W})-\Phi_f(\ind_W)$, $\tilde{\nu}$ is also tangent to $\ind_D$.
\end{proof}

Proposition \ref{prop: nu is tangent to D} and Corollary \ref{coro:intersection formula for Euobs} imply that $\Eu(\tilde{\nu},\Psi_f(\ind_W),0)$ and $\Eu(\tilde{\nu},\ind_D,0)$ are well-defined numbers and can be computed by microlocal intersection formulas.

\section{Microlocal interpretations of indices of logarithmic vector fields}
\label{sec; microlocalintersectionformula}
Let $\tilde{\nu}$ be a continuous vector field on $W$ tangent to $X$, let $\nu=\tilde{\nu}\vert_D$ and asuume $\Sing(\nu)=\{ 0\}$. After recalling the definitions of   the Schwartz index, the GSV index of $\nu$ and the logarithmic index of $\tilde{\nu}$, we will show in some cases these indices can be computed by microlocal intersection formulas.

\subsection{The Schwartz indices for vector fields on isolated singularities}  
The Schwartz index was introduced by  Marie-H\'el\`ene  Schwartz in \cite{MR212842} in her  generalization of Poincar\'e-Hopf theorem to singular spaces (see also \cite{BSS87}\cite{MR1096495}).  Here we  recall its construction on isolated singularities. The general construction  is more complicated  and we refer  to \cite[\S 2.4]{BSS87} for more details.

Assume that $X$ has at most an isolated singularity at $0$. 
Let $B_\epsilon=B_\epsilon(0)$ be a small ball  centered at $0$ and let $V_\epsilon= X\cap B_\epsilon$. This is a cone over the link $K_\epsilon:=S_\epsilon\cap X$. 
We may take a continuous radial vector field $\nu_{rad}$ on $D_\epsilon$ such that $\nu_{rad}$ has  only isolated singularity  at $\{0\}$ and is transverse to all the sphere boundaries $S_\delta$ for $\delta< \epsilon$. This vector field $\nu_{rad}$ extends to a continuous vector field $\tilde{\nu}_{rad}$ on $W$.
 
By  \cite[Theorem 1.1.2]{BSS87},  for  some $\epsilon_1<\epsilon$   we may form a continuous vector field $\omega$ 
on the   cylinder $C$ in $D$ with boundary formed by links $K_\epsilon$
 and  $K_{\epsilon_{1}}$, 
such that 
\begin{enumerate}
\item $\omega$ restricts to $\nu_{rad}$  on the link $K_{\epsilon_{1}}$ .
\item $\omega$ restricts to $\nu$ on the link $K_\epsilon$.
\item $\omega$ has finitely many isolated singularities in the interior of the cylinder $C$.
\end{enumerate}
  
\begin{defi}
Assume that $X$ has at most an isolated singularity at $0$. 
The Schwartz index at $0$ of $\nu$ is defined as 
\[
\Sch(\nu, X, 0):=\Sch(\nu_{rad}, X, 0)+\Ind_{\PH}(\omega, X) \/,
\]
where $\Ind_{\PH}(\omega, X):=\sum_{x\in \Sing(\omega)} \Ind_{\PH}(\omega, x)$ is the total Poincar\'e-Hopf index and the Schwartz index of any radial vector field $\nu_{rad}$ is defined to be $1$, i.e., 
$\Sch(\nu_{rad}, X, 0):=1 \/.$
\end{defi}

\begin{exam} 
if $\nu$ is everywhere transverse to $K_\epsilon$, then we have $\Sch(\tilde{\nu}, X, 0)=1 \/.$  
\end{exam} 

The vector field $\omega$ extends to a continuous vector field on $X$ agreeing with $\nu$ outside $X_\epsilon$ and agreeing with $\nu_{rad}$ inside $X_{\epsilon_1}$ and has singularity locus  $\Sing(\omega)\cup \{0\}$  in $B_\epsilon$.  This vector field further extends to a continuous vector field $\tilde{\omega}$ on $W$.  It is a continuous perturbation of the initial vector field $\tilde{\nu}$. 
 
The standard Hermitian metric $\langle\cdot,\cdot\rangle$ transforms the vector fields $\tilde{\nu}$, $\tilde{\nu}_{rad}$ and $\tilde{\omega}$ into continuous sections $\tilde{\sigma}$, $\tilde{\sigma}_{rad}$ and $\tilde{\rho}$ of $T^*W$. The sections $\tilde{\sigma}(W)$ and $\tilde{\sigma}_{rad}(W)$ only intersect  $T_X^*W$ at $(0,0)$, while the intersection of $\tilde{\rho}(W)$ and $T^*_XW$ is discrete and confined inside $B_\epsilon$.
 
\begin{lemm}
$\sharp_0 \left([\textup{CC}(\ind_X)]\cdot [\tilde{\sigma}_{rad}(B_\epsilon)]\right) =1=\Sch(\nu_{rad}, X, 0)\/.$
\end{lemm}
\begin{proof}
Let $r(x)=\langle x,x\rangle$ be the square of the distance to $0\in W$. We have
\[
\sharp_0 \left([\textup{CC}(\ind_X)]\cdot[dr(B_\epsilon)]\right)=(\Eu^\vee \circ\textup{CC}(\ind_X))(0)=\ind_X(0)=1
\]
by the definition of the dual Euler transformation \cite[Corollary 5.0.1]{MR2031639}. On the other hand, it is easy to find a deformation $\tilde{\sigma}_t$ of continuous sections such that $\tilde{\sigma}_0=\tilde{\sigma}_{rad}$ and $\tilde{\sigma}_1=dr$ and $|T^*_XW|\cap |\tilde{\sigma}_t(W)|=(0,0)$ for any $t\in [0,1]$. For example, by integrating the vector field $\tilde{\sigma}_{rad}$ we can obtain a $C^1$ function $g$ whose gradient vector field is $\tilde{\sigma}_{rad}$. The function $g$ has the property that $g\geq 0$ and $g^{-1}(0)=\{0 \}$. We can then define $g_t=(1-t)g+tr$ and let $\tilde{\sigma}_t=dg_t$. Finally, by conservation of intersection number under continous deformation we have
\[
\sharp_0 \left([\textup{CC}(\ind_X)]\cdot [\tilde{\sigma}_{rad}(B_\epsilon)]\right) =\sharp_0 \left([\textup{CC}(\ind_X)]\cdot[dr(B_\epsilon)]\right)=1.
\]
\end{proof}

We have the following microlocal interpretation of the Schwartz index. 
\begin{prop}
\label{prop; microSchFormula}
Suppose $X$ has at most an isolated singularity at $0$, then
\[
\Sch(\nu, X, 0)=\sharp_0 \left( [\textup{CC}(\ind_X)]\cdot [\tilde{\sigma}(W)] \right) \/.
\]
\end{prop}
\begin{proof} 
Since $\tilde{\omega}$ is a perturbation of $\tilde{\nu}$, $\tilde{\rho}$ is a perturbation of $\tilde{\sigma}$. By conservation of intersection number we have 
\[
\sharp_0 \left([\textup{CC}(\ind_X)]\cdot [\tilde{\sigma}(W)]\right)
=
\sum\sharp_x \left([\textup{CC}(\ind_X)]\cdot [\tilde{\rho}(W)]\right)   
\]
where the summation runs through all $x\in \Sing(\tilde{\omega})\cap D=\{0\}\cup \Sing(\omega)$. Denote the ball of radius $\eta$ centered at $x$ by $B_{\eta,x}$. When $0<\eta \ll \epsilon$, we have
\begin{align*}
\sharp_0 \left([\textup{CC}(\ind_X)]\cdot [\tilde{\rho}(W)]\right)  
=& \sharp_0 \left([\textup{CC}(\ind_X)]\cdot [\tilde{\sigma}_{rad}(B_{\epsilon})]\right)  + 
(-1)^{n-d}\sum_{x\in \Sing(\omega)} \sharp_x \left([T^*_XW]\cdot [\tilde{\rho}(B_{\eta,x})] \right) \\ 
=&  \Sch(\nu_{rad}, X, 0)+  (-1)^{n} \sum_{x\in \Sing(\omega)} \sharp_x \left( [X]\cdot  [\omega(B_{\eta,x}\cap X)] \right) \\
=&  \Sch(\nu_{rad}, X, 0)+ \Ind_{\PH}(\omega, X).
\end{align*}
The second equality is due to the Nash-conormal correspondence as in Theorem~\ref{theo:microlocal for Euobs}, and  that the  Nash blowup is isomorphic to $X$ around any $x\in \Sing(\omega)$ since $\Sing(\omega)$ lies in $X_{sm}=X\setminus\{0\}$. 
The third equality is due to Example~\ref{exam; microPH}. The proposition then follows from the definition of the Schwartz index. 
\end{proof}

The Schwartz index $\Sch(\nu,V,0)$ is defined quite generally for any complex analytic germ $(V,0)\subset (\CC^n,0)$ and any stratified vector field $\nu$ on a local representation $V$ of $(V,0)$ with an isolated singularity at $0$ (See \cite[Definition 2.4.2]{BSS87}). The assumption that the vector field is stratified implies that $\tilde{\nu}$ is tangent to $\ind_V$, so $\Eu(\tilde{\nu},\ind_V,0)$ is well-defined and is equal to $\sharp_0 \left(\textup{CC}(\ind_V)\cdot [\tilde{\sigma}(W)]\right)$ by Corollary \ref{coro:intersection formula for Euobs}. It is presumably true that $\Eu(\tilde{\nu},\ind_V,0)=\Sch(\nu,V,0)$ but we didn't try to prove it seriously. So we put it here as a conjecture.

\begin{conj}
Let $\tilde{\nu}$ be a continuous vector field on $W$ with an isolated singularity at $0$ inducing a stratified vector field on $V$. Then
\[
\Sch(\nu,V,0)=\Eu(\tilde{\nu},\ind_V,0).
\]
\end{conj}

\subsection{The GSV indices for vector fields on hypersurface singularities}  
\label{sec; microlocalGSV}
The GSV index for vector fields on isolated  hypersurface  singularities was introduced in \cite{GSV}  by X. Gomez-Mont, J. Seade and A. Verjovsky. The original definition concerns the obstruction to extending  $\nu$ and $\overline{\textup{grad}}f$ to a local $2$-frame and we refer to \cite{GSV} for the precise definition. Later, the construction was extended to stratified vector fields on any complex hypersurface. We briefly recall the construction below and the readers may refer to \cite[\S 3.5.2]{BSS87} for more details.
 
Let $X=D$ be a reduced complex hypersurface defined by a holomorphic function $f$ on $W$. Take a Whitney stratification $D_\alpha$ of $D$ satisfying the Thom $w_f$ condition. Let $\tilde{\nu}$ be a vector field on $W$ inducing a stratified vector field $\nu=\tilde{\nu}\vert_D$ on $D$ with $\Sing(\nu)=\{0\}$. Let $B_\epsilon=B_\epsilon(0)$ be a small ball  centered at $0$ with boundary $S_\epsilon$ and let $F_t=B_\epsilon\cap f^{-1}(t)$
be the local Milnor fiber of $f$ at $0$. This is a $2n$-dimensional manifold with boundary $\partial F_t = F_t\cap S_\epsilon$. The contraction map from the Milnor tube to the singular fibre $F_0$ induces a continuous map $\pi:F_t\to F_0$ which is surjective and is $C^\infty$ over the smooth part of $F_0$. The Thom $a_f$ condition implies the existence of a local Milnor fibration (cf. \cite[p. 755 and Remark 10.4.13]{MS22}),
and one can define a vector field $\nu_t$ on the Milnor fiber $F_t$ such that $\pi_*(\nu_t)=\nu$, and the stronger $w_f$ condition implies that $\nu_t$ is continuous. Since $\Sing(\nu)=0$, $\nu_t$ is nonsingular in a collar neighbourhood of $\partial F_t$ when $t$ is sufficiently close to 0.

\begin{defi}[{\cite[Definition 3.5.1]{BSS87}}]
The $GSV$ index of $\nu$ on $D$ at $0$, denoted by $\GSV(\nu, D, 0)$, is the total Poinacr\'e-Hopf index of $\nu$ on $F_t$:
\[
\GSV(\nu, D, 0)=\Ind_{\PH}(\nu_t, F_t)=\int [F_t]\cdot [\nu_t(F_t)] \/.
\]  
The last integration makes sense since the intersection of $F_t$ and $\nu_t(F_t)$ (insider $TF_t$) does not meet $\partial F_t$ so is supported in a compact subset of $F_t$. 
\end{defi}

So far we see two competing conditions defining what it means for a vector field $\tilde{\nu}$ to be tangent to a complex hypersurface $D$. 
\begin{itemize}
\item[($\ast$)] The vector field $\tilde{\nu}$ is holomorphic and tangent to $D_{sm}$, i.e. $\tilde{\nu}\in H^0(W,\Der_W(-\log D))$.
\item[($\ast\ast$)] The vector field $\nu=\tilde{\nu}\vert_D$ is continuous, and there exists a stratification $\{D_\alpha\}$ of $D$ which is both Whitney regular and Thom $w_f$-regular such that $\nu$ is tangent to $D_\alpha$ for every $\alpha$.
\end{itemize}

We need condition ($\ast$) to define $\Eu(\tilde{\nu},\Psi_f(\ind_W),0)$ and condition ($\ast\ast$) to define $\GSV(\nu, D, 0)$. Clearly a holomorphic $\tilde{\nu}$ satisfying ($\ast\ast$) automatically satisfies ($\ast$). Conversely, given a vector field $\tilde{\nu}$ satisfying ($\ast$), we know in at least two cases the logarithmic stratification of $D$ satisfies all requirements in $(\ast\ast)$.

\begin{prop}\label{prop; stratification}
The logarithmic stratification of any holonomic divisor is a Whitney stratification. If $D$ is strongly Euler homogeneous, then the pair $(W\setminus D, D_\beta)$ satisfies the $w_f$ condition for any logarithmic stratum $D_\beta$. If $D$ has an isolated singularity at $0$, the pair $(W\setminus D, D\setminus 0)$ satisfies $w_f$ condition in a neighbourhood of $0$.

\end{prop}

\begin{proof}
For the statement that the logarithmic stratification of a holonomic divisor satisfies the Whitney conditions, one can find an argument in \cite{MR3260143} involving the canonical Whiteny stratication of complex analytic varieties. However, one can also explain why Whitney conditions are satisfied by the following simple observation. Let $D_\alpha, D_\beta$ be two adjacent logarithmic strata with $D_\beta\subset \overline{D_\alpha}$. At any $x\in D_\beta$, we can find $d_\beta$ vector fields $\tilde{\nu}_1,\ldots,\tilde{\nu}_{d_\beta} \in \textup{Der}_{W,x}(-\log D)$ such that $\tilde{\nu}_1(x),\ldots,\tilde{\nu}_{d_\beta}(x)$ are linearly independent and they span $T_xD$. For any $y\in D_\alpha$, denote by $T_y$ the vector space spanned by $\tilde{\nu}_1(y),\ldots,\tilde{\nu}_{d_\beta}(y)$. It is clear that $\dim T_y=d_\beta$ when $y\in D_\alpha$ is sufficiently close to $x$. Moreover, the distance between $T_xD_\beta$ and $T_y$ (given by the sine value of the angle between the two vector spaces, see for example \cite{MR4261554} p.250) is a continuous function of both $x$ and $y$, since $T_xD_\beta$ and $T_y$ are generated by a same set of vector fields at different locations. Therefore, there exists an open neibourhood $U_x$ of $x$ in $D$ and a constant $C$ such that
\[
d(T_xD_\beta,T_y) \leq C|x-y|
\] 
whenever $y\in D_\alpha\cap U_x$. By the definition of logarithmic stratification, we also have $T_y \subset T_yD_\alpha$. So
\[
d(T_xD_\beta,T_yD_\alpha) \leq d(T_xD_\beta,T_y) \leq C|x-y|,
\] 
which means that the pair $(D_\alpha,D_\beta)$ satisfies the $w$ condition. One famous theorem of B.Tessier asserts that for complex analytic stratifications the Whitney $w$ condition and $b$ conditions are equivalent (see for example \cite{MR4261554} p.251). The holonomicity condition is only used in the above proof to assure the logarithmic strata are locally finite.

The proof of the second statement is similar. Let $x\in D$ be any point and let $D_\beta$ be the stratum containing $x$. Since $D$ is strongly Euler homogneneous, there exists a local equation $g$ for $D$ at $x$ and a strongly Euler homogeneous vector field $\chi\in  \mathfrak{m}_{W,x}\textup{Der}_{W,x}(-\log D)$ such that $\chi(g)=g$. We have a direct sum decomposition
\[
\textup{Der}_{W,x}(-\log D) \equiv \sO_{W,x}\cdot \chi \oplus \textup{Der}_{W,x}(-\log g)
\]
where $\textup{Der}_{W,x}(-\log g)$ consists of derivations $\delta$ such that $\chi(g)=0$. The $d_\beta$ vector fields $\tilde{\nu}_1,\ldots,\tilde{\nu}_{d_\beta}$ which span $T_xD_\beta$ at $x$ can be chosen from $\textup{Der}_{W,x}(-\log g)$. Let $y\in W\setminus D$. The tangent space to $g^{-1}(g(y))$ at $y$ is spanned by the set $\{ \delta(y) \ \vert \ \delta \in \textup{Der}_{W,x}(-\log g)\}$ which contains the linearly independent vectors $\tilde{\nu}_1(y),\ldots,\tilde{\nu}_{d_\beta}(y)$ provided $y$ is sufficiently close to $x$. An argument as in the proof of the first statement shows that the pair $(W\setminus D, D_\beta)$ satisfies the $w_g$ condition. Finally, since there exists some $h\in \sO_{W,x}^{*}$ such that $f=hg$, the relation $df=hdg+gdh$ shows that the $w_g$ condition and $w_f$ condition at $x$ are equivalent.

The last statement about the isolated singularity case is trivial.
\end{proof}

\begin{prop}
\label{prop; microGSVFormula}
If $\tilde{\nu}$ is holomorphic and satisfies condition $(\ast\ast)$, then 
\[
\GSV(\nu, D, 0)=\sharp_0 \left( [\textup{CC}(\Psi_f(\ind_{W}))]\cdot [\tilde{\sigma}(W)] \right)=\Eu(\tilde{\nu},\Psi_f(\ind_{W}),0).
\]
\end{prop}

\begin{proof}
Using the orthogonal projection $p:TW\vert_{F_t} \to TF_t$ defined by the standard Hermitian metric, we obtain another vector field $\nu'_t=p(\tilde{\nu}\vert_{F_t})$ on $F_t$. It is clear that $\nu'_t$ is a small perturbation of $\nu_t$ in a small collar neighbourhood of $\partial F_t$. Therefore we have
\[
\GSV(\nu, D, 0)=\int [F_t]\cdot [\nu_t(F_t)]=\int [F_t]\cdot [\nu'_t(F_t)]= (-1)^{n-1}\int [T^*_{F_t}W]\cdot [\tilde{\sigma}(W)].
\]

On the other hand,  
by the specialization formula for characteristic cycle for $\ind_W$ (see \cite[Thm. 10.3.55]{MS22} for example), taking the limit $t$ approaching to $0$ we have
\[
\GSV(\nu, D, 0)= (-1)^{n-1}\lim_{t\to 0}\int [T^*_{F_t}W]\cdot [\tilde{\sigma}(W)]=\int [\textup{CC}(\Psi_f(\ind_{W}))]\cdot [\tilde{\sigma}(W)].
\]
The last integral is well defined and is equal to $\sharp_0 \left( [\textup{CC}(\Psi_f(\ind_{W}))]\cdot [\tilde{\sigma}(W)] \right)$ since by Proposition \ref{prop: nu is tangent to D} the intersection of $\textup{CC}(\Psi_f(\ind_{W}))$ and $\tilde{\sigma}(W)$ takes place only at $(0,0)$. The last expression is equal to $\Eu(\tilde{\nu},\Psi_f(\ind_{W}),0)$ by Corollary \ref{coro:intersection formula for Euobs}.

\end{proof}

From the microlocal point of view, it seems very natural to define a GSV index for any holomorphic vector field $\tilde{\nu}$ satisfying condition $(\ast)$ by $\GSV(\nu, D, 0):=\Eu(\tilde{\nu},\Psi_f(\ind_{W}),0)$. However, the relation to the obstruction theory of vector fields in this case looks somewhat obscure.

\subsection{The Microlocal logarithmic Indices of vector fields}
We first review the logarithmic index introduced by Aleksandrov in \cite{AA05}.
In this subsection we assume that $(D, 0)\subset (W, 0)$ is the germ of a reduced hyprsurface, where $W$ is an open subset of $\CC^n$. Let $\tilde{\nu}$ satisfy condition ($\ast$) above. In earlier part of the paper we only need $\Sing(\nu)=\{ 0 \}$, but here we make the stronger assumption that $\Sing(\tilde{\nu})=\{0\}$. We fix a Hermitian metric and denote by $\tilde{\sigma}$ the dual differential form of $\tilde{\nu}$. 

Let $\Omega^1_W(\log D)$ be the $\sO_W$-dual of $\Der_W(-\log D)$ and let $\Omega_W^q(\log D)$ be the $q$-th wedge of $\Omega^1_W(\log D)$, the module of meromorphic $q$-forms with poles along $D$. 
The contraction with the logarithmic vector field $\tilde{\nu}$ induces the following complex 
of $\sO_{W,0}$-modules:
\[
\begin{tikzcd}
0 \arrow[r] & \Omega^n_{W,0}(\log D) \arrow[r, "i_{\nu}"] &\Omega^{n-1}_{W,0}(\log D)   \arrow[r, "i_{\nu}"] & \cdots \arrow[r, "i_{\nu}"] & \Omega^{1}_{W,0}(\log D)  \arrow[r, "i_{\nu}"] &  \sO_{W,0} \arrow[r] & 0
\end{tikzcd}
\] 
We denote this complex by $(\Omega_{W,0}^\bullet(\log D), i_{\tilde{\nu}})$.
Since $\Sing(\tilde{\nu})=\{0\}$,  by \cite[Lemma 1]{AA05}
 the  homology  groups of $(\Omega_{W,0}^\bullet(\log D), i_{\tilde{\nu}})$ are   finite dimensional $\CC$-vector spaces. 
\begin{defi}
 The logarithmic index of $\tilde{\nu}$ along $D$ at  $0$ is defined by 
 \[
 \Ind_{\log}(\tilde{\nu}, D, 0):=\chi \left(\Omega_{W,0}^\bullet(\log D), i_{\tilde{\nu}} \right) =\sum_{i=0}^n (-1)^i \dim_\CC H_i(\Omega_{W,0}^\bullet(\log D), i_{\tilde{\nu}}) \/.
 \]
\end{defi}

We now provide geometric interpolation  for this logarithmic index  in two cases: when $D$ has an isolated singularity at 
$0$ or when $D$ is holonomic, free, and strongly Euler homogeneous. 

First we consider the case of isolated singularity. Let $\tau_D(0)$ and $\mu_D(0)$ denote the Tjurina number and the Milnor number of $D$ at $0$, respectively.
\begin{prop}
\label{prop; micrologFormulaiso}
If the germ $(D, 0)$ has  an isolated singularity at $0$, we have 
\[
\Ind_{\log}(\tilde{\nu}, D, 0)=\sharp_0 \left([\textup{CC}(\ind_{W\setminus D})]\cdot [\tilde{\sigma}(W)]  \right)+(-1)^{n-1}\left(\tau_D(0)-\mu_D(0) \right) \/.
\]
\end{prop}
\begin{proof}
Let $f$ be a defining equation of $D$ in $W$. For $q=0,1, \cdots ,n$ we consider the $\sO_{D,0}$-module of differential forms $\Omega_{D, 0}^q= \Omega_{W, 0}^q/(f\cdot \Omega_{W,0}^q+df\wedge \Omega_{W,0}^{q-1})$. The contraction with the restricted vector field $\nu$ induces a homological complex of $\sO_{D,0}$-modules 
\[
(\Omega_{D,0}^\bullet, i_\nu): 
\begin{tikzcd}
0 \arrow[r] & \Omega^n_{D,0}  \arrow[r, "i_{\nu}"] &\Omega^{n-1}_{D,0}  \arrow[r, "i_{\nu}"] & \cdots \arrow[r, "i_{\nu}"] & \Omega^{1}_{D,0} \arrow[r, "i_{\nu}"] &  \sO_{D,0} \arrow[r] & 0
\end{tikzcd}
\]
and its truncation
\[
(\overline{\Omega}_{D,0}^\bullet, i_\nu): 
\begin{tikzcd}
0 \arrow[r] & \Omega^{n-1}_{D,0}  \arrow[r, "i_{\nu}"] &\Omega^{n-2}_{D,0}  \arrow[r, "i_{\nu}"] & \cdots \arrow[r, "i_{\nu}"] & \Omega^{1}_{D,0} \arrow[r, "i_{\nu}"] &  \sO_{D,0} \arrow[r] & 0
\end{tikzcd} \/.
\]
The homology  groups of both complexes are  finite dimensional $\CC$-vector spaces and we can assign two homological indices
\[
\Ind^{GM}_{hom}(\nu, D, 0)
:=\chi(\overline{\Omega}_{D,0}^\bullet, i_\nu)\/; \quad 
\Ind^{A}_{hom}(\nu, D, 0):=
\chi(\Omega_{D,0}^\bullet, i_\nu) \/.
\] 

Since $D$ has only isolated singularity, 
combining   Proposition~\ref{prop; microGSVFormula} with \cite[Theorem 3.5]{MR1642757} we have 
\[
\Ind^{GM}_{hom}(\nu, D, 0)=\GSV(\nu, D, 0)=\sharp_0 \left([\textup{CC}(\Psi_f(\ind_W))]\cdot [\tilde{\sigma}(W)]  \right) \/.
\]

On the other hand, by definition we see that $\Omega_{D,0}^n=\sO_{W, 0}/(f, \cJ_f)$ is exactly the Tjurina algebra  of $f$,
where $\cJ_f$ denotes the Jacobian ideal of $f$.
 Therefore by   \cite[P8]{MR1642757}   we see that
\[
 \Ind^{A}_{hom}(\nu, D, 0)-\Ind^{GM}_{hom}(\nu, D, 0)
 = (-1)^n \dim_\CC \sO_{W, 0}/(f, \cJ_f)=(-1)^n \tau_D(0)  \/.
\]
  By \cite[Proposition 1]{AA05} we then have 
\begin{equation}
\label{eq; logandtau}
\Ind_{\log}(\tilde{\nu}, D, 0) 
=\Ind_{PH}(\tilde{\nu}, 0)-\GSV(\nu, D, 0)-(-1)^n \tau_D(0) \/.
\end{equation}

Since $D$ has isolated singularity we have $\Psi_f(\ind_W)=\ind_D+(-1)^{n-1}\mu_D(0)\cdot \ind_{\{0\}}$. Therefore 
\begin{align*}
\sharp_0 \left([\textup{CC}(\ind_{W\setminus D})]\cdot [\tilde{\sigma}(W)]  \right)
=&\sharp_0 \left([\textup{CC}(\ind_{W})]\cdot [\tilde{\sigma}(W)]  \right)-\sharp_0 \left([\textup{CC}(\ind_{D})]\cdot [\tilde{\sigma}(W)]  \right) \\
(\text{by Example~\ref{exam; microPH}}) \  =
&  \Ind_{PH}(\title{\nu}, 0)
-\sharp_0 \left([\textup{CC}(\Psi_f(\ind_W))]\cdot [\tilde{\sigma}(W)]  \right)    -(-1)^n \mu_D(0)   \\
(\text{by Proposition~\ref{prop; microGSVFormula}}) \  
=&  \Ind_{PH}(\title{\nu}, 0)
-\GSV(\nu, D, 0) -(-1)^n \mu_D(0)   \\
(\text{by equation \eqref{eq; logandtau}})  \ 
=& \Ind_{\log}(\tilde{\nu}, D, 0) +(-1)^n\left(\tau_D(0)-\mu_D(0)\right) \/. \qedhere
\end{align*}
\end{proof}
\begin{rema}
We note that the term homological index is used in the literature to refer to both indices $\Ind^{GM}_{hom}(\nu, D, 0)$ and $\Ind^{A}_{hom}(\nu, D, 0)$. Readers should carefully check which one is being used when applying formulas.
\end{rema}

We have the following immediate corollary. 
\begin{coro}
\label{coro; microcriteriaofWH}
Assume that $D$ has an isolated singularity at $0$. The  following statements are equivalent.
\begin{enumerate}
	\item $D$ is quasi-homogeneous at $0$, i.e., $f\in \cJ_f$.
	\item There exists a holomorphic vector field $\tilde{\nu}$ in $W$, with an isolated singularity at $0$ and logarithmic to $D$ at $0$, such that 
\[
\Ind_{\log}(\tilde{\nu}, D, 0)=\sharp_0 \left([\textup{CC}(\ind_{W\setminus D})]\cdot [\tilde{\sigma}(W)]  \right) \/.
\]
\item For all holomorphic vector fields $\tilde{\nu}$ in $W$ with an isolated singularity at $0$ and logarithmic to $D$ at $0$, we have 
\[
\Ind_{\log}(\tilde{\nu}, D, 0)=\sharp_0 \left([\textup{CC}(\ind_{W\setminus D})]\cdot [\tilde{\sigma}(W)]  \right) \/.
\]
\end{enumerate}
\end{coro}

Next we turn to the other case where $(D, 0)$ is the germ of a free divisor, i.e., $\Der_{W,0}(-\log D)$ is a free $\sO_{W,0}$-module. Let $\{\chi_1, \chi_2, \cdots ,\chi_n \}$ be an $\sO_{W, 0}$-basis of $\Der_{W, 0}(-\log D)$ so that
\[
\tilde{\nu}= \sum_{i=1}^n \alpha_i\cdot \chi_i \in (\mathfrak{m}_{W, 0}\Der_{W, 0})\cap \Der_{W, 0}(-\log D).
 \]

It is clear that the common zeroes of the coefficients $\alpha_1,\ldots,\alpha_n$ is contained in $\Sing(\tilde{\nu})=\{0\}$. So either $(\alpha_1,\ldots,\alpha_n)=\sO_{W,0}$ or $\alpha_1,\ldots,\alpha_n$ form a regular sequence.

\begin{prop}[{\cite[Corollary 2]{AA05}}]
\label{prop; logindex}  If $\tilde{\nu}$ has an isolated singularity at $0$, then 
\[
\Ind_{\log}(\tilde{\nu}, D, 0)=\dim_\CC \sO_{W,0}/(\alpha_1, \cdots ,\alpha_n) \/.
\]
\end{prop}

\begin{prop}
\label{prop; micrologFormulafree}
If $D$ is a holonomic strongly Euler homogeneous free divisor, then 
\[
\Ind_{\log}(\tilde{\nu}, D, 0)=  \sharp_0 \left([\textup{CC}(\ind_{W\setminus D})]\cdot [\tilde{\sigma}(W)]  \right)   \/.
\]
\end{prop} 

\begin{proof}
First, note that Proposition \ref{prop: nu is tangent to D} implies that $\tilde{\nu}$ is tangent to $\ind_{W\setminus D}$ so the intersection number $\sharp_0 \left([\tilde{\sigma}(W)]\cdot [\textup{CC}(\ind_{W\setminus D})]\right)$ is well-defined. Composing $j:T^*W \to T^*W(\log D)$ with the section $\tilde{\sigma}: W \to T^*W$, we get a section $j\circ \tilde{\sigma}: W \to T^*W(\log D)$. By \cite[Theorem 4.6 and Corollary 5.10]{LZ24-1} We have
\[
\sharp_0\left([\textup{CC}(\ind_{W\setminus D})]\cdot [\tilde{\sigma}(W)]\right) = (-1)^{n}\sharp_0\left( [j^{-1}(W)]\cdot [\tilde{\sigma}(W)]\right) =(-1)^n\sharp_0\left([W]\cdot [j\circ \tilde{\sigma}(W)]\right).
\]
The standard Hermittian metric on $\mathbb{C}^n$ gives us an $\mathbb{R}$-linear isomorphism $TW(-\log D) \cong T^*W(\log D)$. Under this isomorphism $j\circ \tilde{\sigma}$ corresponds to $\tilde{\nu}$ regarding as a holomorphic section of $TW(-\log D)$. 
Therefore we have
\[
(-1)^n\sharp_0\left([j\circ \tilde{\sigma}(W)] \cdot [M]\right)=\sharp_0\left([W]\cdot\tilde{\nu}(W)\right)=\dim_\CC \sO_{W,0}/(\alpha_1, \cdots ,\alpha_n)
\]
where the last equality follows from \cite[Proposition 7.1]{MR1644323}.
\end{proof}

\begin{exam}
When $n=2$ and $(D, 0)$ is the germ of a  curve singularity, $\Der_{W, 0}(-\log D)$ is a free $\sO_{W,0}$-module and  $D$ is a holonomic free divisor.  
If moreover D is weighted homogeneous, then Proposition~\ref{prop; micrologFormulaiso} and Proposition~\ref{prop; micrologFormulafree}  lead to the same conclusion since $\tau_D(0)=\mu_D(0)$.
\end{exam}

\subsection{Example: $\chi$-number,multiplicity and polar intersection number of planar foliations}
\label{sec; foliationontheplane}
We will consider foliations on the plane  and give a microlocal explanation of the $\chi$-number, the multiplicity and the polar multiplicity considered by A.~Fern\'andez-P\'erez et~al.  in \cite{MR4886053}. Since the original definitions of these indices are quite involved and fall outside the scope of this paper, we will skip the details and only state our explanation, referring the reader to \cite{MR772129,MR4886053,MR4012805} for  details. 

Let $\cF$ be a foliation on a planar region $W\subset\CC^2$  given by the vector field $\tilde{\nu}$ with an isolated singularity at $0\in W$. Let $\tilde{\sigma}$ be the section of $T^*W$ induced by any Hermitian metric on $W$.

By a separatrix we mean a reduced irreducible curve germ $(B, 0)$ such that $\cF$ is logarithmic along $B$.  The sepratrices of $\cF$ are either isolated or dicritical. Any given foliation $\cF$ admits only finitely many isolated sepatrices but may adimit infinitely many dicritical ones.  
We denote the set of isolated sepratrices by $\{B_1, B_2, \cdots ,B_r\}$ and the set of dicritical  sepratrices by 
$\{ C_\alpha|\alpha\in I\}$. A balanced divisor of separatrices for $\cF$ is a divisor of curve germs of the form
\[
\cB=\sum_{i=1}^r B_i +\sum_{\alpha\in J } m_\alpha C_\alpha 
\]
where $J\subset I$ is a finite subset and the coefficients $m_\alpha$ satisfy certain balanced conditions (see  \cite[Definition 2.3]{MR4012805} for details). 

\begin{defi}
Given any divisor $\cB=\sum_B a_B B$ with only finitely many non-zero $a_B$, we define the constructible functions $\ind_{\cB}$ ane $\ind_{W\setminus \cB}$ as
\[
\ind_{\cB}=  
\sum_B a_B\ind_{B}   - (\deg \cB -1 )\ind_{\{0\}}, \quad \ind_{W\setminus \cB}=\ind_W-\ind_{\cB}
\]
where $\deg \cB = \sum_B a_B$. 
\end{defi}

\begin{prop} 
Let $C$ be any  separatrix of $\cF$. Let $\mu_0(\cF, C)$ 
and $i_0(\cP^{\cF}, C)$ be 
 the multiplicity of $\cF$ along $C$ and the  polar intersection number of $\cF$ with respect to $C$ defined in 
\cite[\S 4]{MR4886053}. We have
\[
\mu_0(\cF, C)= \Sch(\tilde{\nu}, C, 0) =
\sharp_0 \left( [\textup{CC}(\ind_C)]\cdot [\tilde{\sigma}(W)] \right) \/, \quad i_0(\cP^{\cF}, C)=  
\sharp_0 \left( [\textup{CC}(\Eu_C)]\cdot [\tilde{\sigma}(W)] \right) \/.
\]
\end{prop} 
\begin{proof}
Let $\gamma(t)=(x(t), y(t))$ be  any Puiseux parametrization of $C$ such that $x(t)\neq 0$. Let $f(x,y)$ be the defining equation of $C$, 
by \cite[Proposition 5.7]{MR4886053} and   \cite[Lemma 2.1]{Ploski95} we have 
\[
\mu_0(\cF, C) = \GSV (\tilde{\nu}, C, 0) + \left(\textup{ord}_0 \partial_y f(\gamma(t)) - \textup{ord}_0(x'(t)) \right) =\GSV (\tilde{\nu}, C, 0) +  \mu_C(0)  \/.
\] 
Then we have 
$
\mu_0(\cF, C)     = \Sch(\tilde{\nu}, C, 0)  
$ by Remark~\ref{rema:GSV-SCH=mu}.
The second equality is  simply a  restatement of \cite[Proposition  4.5]{MR4886053}.
\end{proof}

Recall that $\mu_0(\cF)=\sharp_0 \left( [\textup{CC}(\ind_{W})]\cdot [\tilde{\sigma}(W)] \right)$. Then \cite[Proposition 4.7]{MR4886053}  can be restated as
\begin{prop} 
Let $\chi_0(\cF)$ be the $\chi$-number of the foliation $\cF$ defined in \cite[\S 3]{MR4886053}. Let  $\cB$ any balanced divisor  of $\cF$, then  we have
\begin{equation}
\label{eq; chinumbermicro}
\chi_0(\cF)=\sharp_0 \left( [\tilde{\sigma}(W)]\cdot [\textup{CC}(\ind_{W\setminus \cB})] \right)  \geq 0
  \/.
\end{equation}
The last inequality is due to  \cite[Proposition 3.1]{MR4886053}.
\end{prop}

With these microlocal intersection formulas in mind, quite a few relations among these indices become obvious. For example, \cite[Corollary 5.3]{MR4886053} can be explained by the identity of constructible functions 
\[
\Psi_f(\ind_W)= \ind_C +\Phi_f(\ind_W) =\Eu_C - (\Eu_C(0) -1-\mu_C(0))\ind_{\{0\}} \/.
\]
and \cite[Corollary (5.6) ]{MR4886053}  follows from the definition of $\ind_\cB$ and $\ind_{W\setminus \cB} $.

When $\cF$ is non-dicritical, there is no dicritical separatrices and  the divisor of all separatrices $D=\cup_{i=1}^r B_i$   is the unique balanced divisor.   This is   a holonomic and free divisor, and we have the following property from Proposition~\ref{prop; micrologFormulafree}.
\begin{prop}
If moreover $D$ is weighted homogeneous at $0$, we have 
\[
\chi_0(\cF) = \Ind_{\log} (\tilde{\nu}, D, 0) \/.
\]
\end{prop}

We close this subsection with the following question.
\begin{ques}
Let $\cB=\sum_{i=1}^r B_i +\sum_{\alpha\in J } m_\alpha C_\alpha $ be a divisor of separatrices of $\cF$. Does the equality $\chi_0(\cF)=\sharp_0 \left( [\tilde{\sigma}(W)]\cdot [\textup{CC}(\ind_{W\setminus \cB})] \right) $ in ~\eqref{eq; chinumbermicro} necessarily imply that $\cB$ is a balanced divisor ?
\end{ques}

\section{Global index formulas for one-dimensional holomorphic foliations}
\label{sec; chernclassoffoliation}
In this section we consider a  reduced hypersurface $D$ on a compact complex manifold $M$ cut by a global section $f\in H^0(M, \sO(D))$ and 
a logarithmic one-dimensional foliation  $\cF$  along $D$ with tangent bundle $L$. Let  $s_\cF$ be the global holomorphic section  of $TM\otimes L^\vee$ corresponding to $\cF$.  We assume that $\cF$ has only isolated singularities.

Fix an Hermitian metric  on $TM\otimes L^\vee$ and identify $TM\otimes L^\vee$ with $T^*M \otimes L$ as real vector bundles on $M$. This Hermitian metric  transforms the holomorphic section $s_\cF$ to a $C^\infty$ section $\sigma_\cF$ of $T^*M \otimes L$. 
Since $\cF$ has only isolated singularities,   $\sigma_\cF$ also has only isolated singularities, i.e. $|\sigma_\cF(M) \cap M|<\infty$. Take an open covering $\{U_\alpha \}$ of $M$ such that $L\vert_{U_\alpha}$ is trivial for every $\alpha$. On each $U_\alpha$, the sections $s_\cF$ and its dual $\sigma_\cF$ can be represented by a holomorphic vector field $\tilde{\nu}_\alpha$ and a $C^\infty$ differential form $\sigma_\alpha$. Let also $D_\alpha=D\cap U_\alpha$.

\begin{defi} 
For any $x\in \Sing(\cF)$ we define the Poincar\'e-Hopf index,
the Schwartz index, the GSV index and the logarithmic index  of the logarithmic foliation $\cF$ along $D$  at $x$ by 
\begin{align*}
\Ind_{\PH}(\cF, x)= \Ind_{\PH}(\tilde{\nu}_\alpha, x)\/,  &\quad
\Sch(\cF, D, x) =   \Sch(\tilde{\nu}_\alpha, D_\alpha, x)\/,  \\
\GSV(\cF, D, x) =   \GSV(\tilde{\nu}_\alpha, D_\alpha, x)  \/, &\quad 
\Ind_{\log}(\cF, D, x) =   \Ind_{\log}(\tilde{\nu}_\alpha, D_\alpha ,x) \/.
\end{align*}
whenever these local indices are defined for $x\in U_\alpha$. 
It is easy to see that all these indices are independent of the choices of local representing vector fields. 
\end{defi}

There are many index formulas in the theory of one-dimensional holomorphic foliation, and a lot of these can be considered as generalizations of the Baum-Bott formula~\eqref{eq0}. The prototypical proof in differential geometry uses the Chern-Weil theory where one finds a differential form to represent the Chern class in question. The local indices can be computed with the help of the Grothendieck residue of the differential form. In constrast, we interpret these local indices as intersection numbers of $\sigma_\cF(M)$ with appropriate jet characteristic cycles. We obtain the global index formulas by deforming $\sigma_\cF(M)$ to the zero section of $T^*M\otimes L$ and apply Aluffi's twsited shadow formula Proposition~\ref{prop; jettwistformula}. Sometimes this strategy gives shorter and more conceptual proofs.

For instance, let's revisit Baum-Bott formula~\eqref{eq0}. It is worth mentioning that there is a one-line proof using the standard intersection theory in algebraic geometry which goes as
\begin{equation}
\label{eq; sumofPH}
\int_M c(TM-L)\cap [M]= \int_M c_n(TM\otimes L^\vee)\cap [M] = \sum_{x\in \Sing(\cF)} \mu_x(\cF) =\sum_{x\in \Sing(\cF)} \Ind_{\PH}(\cF, x)
\end{equation}
where $\mu_x(\cF)=\mu_x(\tilde{\nu}_\alpha)$ if $x\in U_\alpha$. The second equality is the key step and it follows from \cite[Proposition 14.1]{MR1644323} since $s_\cF$ is a regular section of $TM \otimes L^\vee$. 

Instead of using the section $s_\cF$ in the proof above, we may as well use the dual $\sigma_\cF$. This will yield an essentially same proof, but tailored better to the strategy described above. Indeed, since 
\[
\sum^n_{i=0}(-1)^{n-i}c_i(TM)\cap[M]=(-1)^nc(T^*M)\cap [M]=\textup{Shadow}_{T^*M}(\textup{CC}(\ind_M)),
\] 
by Proposition~\ref{prop; jettwistformula} (applied to $\alpha_i=(-1)^{i}c_{n-i}(TM)\cap [M]$) we have 
\[
 \int_M c(TM-L)\cap [M]
 =\sum_{i=0}^n \int_M (-1)^{i}c_1(L)^{i}c_{n-i}(TM)\cap [M]
 =\deg \left([\textup{JCC}(\ind_M)]\cdot [M]  \right)    \/. 
\]
Deformation invariance of the intersection number implies that 
\[
\deg \left([\textup{JCC}(\ind_M)]\cdot [M]  \right) 
=  \deg \left([\textup{JCC}(\ind_M)]\cdot [\tilde{\sigma}_\cF(M)] \right)   
=  \sum_{x\in \Sing(\cF)}\sharp_x \left([\textup{JCC}(\ind_M)]\cdot [\tilde{\sigma}_\cF(M)]\right)  \/.
\]
Finally by Example~\ref{exam; microPH} we have  $\Ind_{\PH}(\cF, x)= \sharp_x \left([\textup{JCC}(\ind_M)]\cdot [\tilde{\sigma}_\cF(M)]\right)$. This completes the alternative proof.

\begin{defi}
\label{defn; **}
We say a logarithmic one-dimensional foliation  $\cF$ along $D$ satisfies   ($\ast\ast$), if for each   $x\in \Sing(\cF)\cap D$  there exists an analytic neighborhood $U_x$ of $x$ such that the local vector field representing   $\cF$  satisfies the condition ($\ast\ast$) in $U_x$.
\end{defi}

Combining
Proposition~\ref{prop; microSchFormula}, Proposition~\ref{prop; stratification}, Proposition~\ref{prop; microGSVFormula}, Proposition~\ref{prop; micrologFormulaiso} and Proposition~\ref{prop; micrologFormulafree} we obtain the following comparison theorem which will be used for deriving global index formulas.
\begin{theo}\ 
\label{theo; comparison}
Let $\cF$ be a logarithmic one-dimensional foliation along $D$ with tangent bundle $L$. Assume $\cF$ has only isolated singularities and $x\in \Sing(\cF)\cap U_\alpha$.
\begin{enumerate}
\item If $D$ has only isolated singularities, then 
\begin{align*}
\Sch(\cF, D, x)  =& \sharp_x \left([\textup{CC}( \ind_D)]\cdot [\sigma_\alpha(U_\alpha)]\right) = \sharp_x \left([\textup{JCC}( \ind_D)]\cdot [\sigma_\cF(M)]\right) \/, \\
\Ind_{\log}(\cF, D, x)=& \sharp_0 \left([\textup{CC}(\ind_{M\setminus D})]\cdot [\tilde{\sigma}(W)]  \right)+(-1)^{n-1}\left(\tau_D(x)-\mu_D(x) \right) \/.
\end{align*}

\item If $\cF$  satisfies  ($\ast\ast$), then 
\[
\GSV(\cF, D, x)  =  \sharp_x \left([\textup{CC}(\Psi_f(\ind_M)]\cdot [\sigma_\alpha(U_\alpha)]\right) = \sharp_x \left([\textup{JCC}(\Psi_f(\ind_M)]\cdot [\sigma_\cF(M)]\right)  \/.
\]
\item If $D$ is a holonomic strongly Euler homogeneous free divisor, then
\[
\Ind_{\log}(\cF, D, x)=  \sharp_x \left([\textup{CC}(\ind_{M\setminus D})]\cdot [\sigma_\alpha(U_\alpha)]\right) =\sharp_x \left([\textup{JCC}(\ind_{M\setminus D})]\cdot [\sigma_\cF(M)]\right)  \/.
\]
\end{enumerate}
\end{theo}  

\begin{coro}[An algebraic formula for the GSV index]\label{coro:algebraicGSV}
If $x\in U_\alpha$ and $D$ has at most an isolated singularity at $x$, then  
\[
\GSV(\cF, D, x)  = \Eu^{\tilde{\nu}_\alpha}_D(x) +1 +(-1)^{n-1}\mu_D(x) -\Eu_D(x).
\]
where $\mu_D(x)$ is the Milnor number of $D$ at $x$, i.e. $(-1)^{n-1}$ multiplied by the reduced Euler characteristic of the Milnor fibre at $x$.
\end{coro}

\begin{proof}
Since $D$ has at most an isolated singularity at $x$, there exists some $k\in\mathbb{Z}$ such that
\[
\Psi_f(\ind_M)=\Eu_D+k\cdot\ind_x.
\]
We may get the number $k$ by the relation $\Psi_f(\ind_M)-\Phi_f(\ind_M)=\ind_D$. Plugging in $x$ we obtain
\[
k=1+\Phi_f(\ind_M)(x)-\Eu_D(x)=1+(-1)^{n-1}\mu_D(x)-\Eu_D(x).
\]
Therefore
\begin{align*}
\GSV(\cF, D, x) &= \sharp_x \left([\sigma_\alpha(U_\alpha)]\cdot [\textup{CC}(\Psi_f(\ind_M)]\right) \\
&= \sharp_x \left(\textup{CC}(\Eu_D)\cdot [\sigma_\alpha(U_\alpha)] \right) +k\cdot \sharp_x\left(\textup{CC}(\ind_x)\cdot [\sigma_\alpha(U_\alpha)] \right) \\
&=\Eu(\tilde{\nu}_\alpha,\Eu_D,x)+ 1+(-1)^{n-1}\mu_D(x)-\Eu_D(x) \/.  \qedhere
\end{align*}
\end{proof}

\begin{rema}
\label{rema; NashEu}
Let $\pi:\tilde{D}\to D$ be the Nash blowup of $D$ and denote the Nash tangent bundle by $\cT$. The foliation $s_\cF\in H^0(M,TM\otimes L^\vee)$ induces a section $\tilde{s}_\cF: H^0(\tilde{D},\cT\otimes \pi^*L)$ since $\cF$ is logarithmic along $D$. Then
\[
\Eu^{\tilde{\nu}_\alpha}_D(x)=\int_{\pi^{-1}(x)} [\tilde{s}_\cF(\tilde{D})] \cap [\tilde{D}].
\]
\end{rema}

\begin{rema}\label{rema:GSV-SCH=mu}
When $D$ has an isolated singularity at $x$, it is proved in \cite[Theorem 3.2.1]{BSS87} that  
\[
\GSV (\cF, D, x)= \Sch (\cF, D, x) +(-1)^{n-1}\mu_D(x) \/.
\]
This follows easily from Theorem \ref{theo; comparison} and $\Psi_f(\ind_M)=\ind_D + \Phi_f(\ind_M)=\ind_D + (-1)^{n-1}\mu_x(D)\cdot\ind_x$.
\end{rema}

The following global index formula generalizes 
\cite[Theorem 2]{CM19} and \cite[Theorem 1]{CM24}, which were stated only for normal crossing divisors.
\begin{coro} 
\label{coro; SEH}
Let $\cF$ be a logarithmic one-dimensional foliation along $D$ with tangent bundle $L$.  
We assume $\cF$ has only isolated singularities.
\begin{enumerate}
\item If  $D$ is holonomic, strongly Euler homogeneous and free, then 
\[
\int_M c(TM(-log D)-L)\cap [M]=\sum_{x\in Sing(\cF)\setminus D}\Ind_{\PH}(\cF, x) + \sum_{y\in Sing(\cF)\cap D} \Ind_{\log}(\cF, D, y) \/.
\]

\item If  $D$ has only isolated singularities and $L$ is very ample, then
\[
\chi(M\setminus D\cup V)=\sum_{x\in Sing(\cF)\setminus D}\Ind_{\PH}(\cF, x) + \sum_{y\in Sing(\cF)\cap D} \Ind_{\log}(\cF, D, y) +(-1)^n \left( \mu_D-\tau_D\right) \/,
\]
where $V$ is  a generic hypersurface in the linear system $|L|$, $\mu_D$ and $\tau_D$ denote the total Tjurina number and the total Milnor number of $D$.
\end{enumerate}
\end{coro} 
\begin{proof} 
Let $U=M\setminus D$ be the complement. Since $s_\cF$  is homotopic to the zero section, we have 
\[
 \deg \left([\textup{JCC}(\ind_U)]\cdot [M] \right)  
=  \deg \left([\textup{JCC}(\ind_U)]\cdot [\sigma_\cF(M)] \right)  \/.
\]
Both conclusions are avatars of this equality.  Since $\cF$ has  only isolated singularities, the RHS equals 
\[
 \sum_{x\in \Sing(\cF)\cap U} \sharp_x \left( [\textup{JCC}(\ind_U)]\cdot [\sigma_\cF(M)] \right)+\sum_{y\in \Sing(\cF)\cap D} \sharp_y \left( [\textup{JCC}(\ind_U)]\cdot [\sigma_\cF(M)] \right) \/.
\] 
 By Example~\ref{exam; microPH}, Proposition~\ref{prop; micrologFormulaiso}  and  Proposition~\ref{prop; micrologFormulafree}, in the first case  this   equals 
\[
\sum_{x\in \Sing(\cF)\cap U} \Ind_{\PH}(\cF, x) +\sum_{y\in \Sing(\cF)\cap D}  \Ind_{\log}(\cF, D, y) \/, 
\]
while in the second case this equals    
\[
\sum_{x\in Sing(\cF)\setminus D}\Ind_{\PH}(\cF, x) + \sum_{y\in Sing(\cF)\cap D} \Ind_{\log}(\cF, D, y) +(-1)^n \left( \mu_D-\tau_D\right) \/.
\] 
For the LHS, in the second case by \cite[Proposition 7.7]{LZ24-1} we have    
\[
\chi(M\setminus D\cup V)
=\deg \left([\textup{JCC}(\ind_U)]\cdot [M] \right)  \/.
\] 
In the first case, 
 by Proposition~\ref{prop; hSEHchernformula} we have
\[
 c_*(\ind_{U})=\textup{Dual-Shadow}_{T^*M} \left([\textup{CC}(\ind_U)]\right)=c(TM(-\log D)) \cap [M] \/.
\]
Therefore
\[
\textup{Shadow}_{T^*M} \left([\textup{CC}(\ind_U)]\right)=\sum^n_{j=0} (-1)^{n-j}c_{j}(TM(-\log D)) \cap [M].
\]

By the   twisted shadow  formula   Proposition~\ref{prop; jettwistformula} we then obtain
\begin{align*}
\int_M c(TM(-\log D)-L)\cap [M] 
=& \sum_{j=0}^n \int_M (-1)^{n-j}c_1(L)^{n-j}c_k(TM(-\log D)) \cap [M] \\
=&   \deg \left([\textup{JCC}(\ind_U)]\cdot [M] \right)   \/. \qedhere
\end{align*}
\end{proof}

We are aiming to generalize \cite[Theorem 7.16]{Suwa}, \cite[Theorem A]{Macha25} and  
\cite[Theorem 1-(III)]{Macha25-Mar} to arbitrary singularities. For this we  assume that the foliation $\cF$  along $D$ satisfies condition ($\ast\ast$), so that the $\GSV$ index is defined. By Proposition~\ref{prop; stratification} this includes the cases where either $D$ has only isolated singularities or $D$ is a holonomic and strongly Euler homogeneous divisor.

Let $\Sing(D)$ be the singular subscheme of $D$ and let $\cJ_D$ be the corresponding ideal sheaf. Let $\tilde{M}:=\textup{Bl}_{\cJ_D}M$ be the blowup of $M$ along $\Sing(D)$ and let $\pi\colon \tilde{M}\to M$ be the blowup morphism. 
We denote by $\sD=\pi^{-1}(D)$  the total transform of $D$ and by $E$  the exceptional divisor of this blowup.   These are Cartier divisors on $\tilde{M}$ and we consider the virtual bundles 
$\cK:=\pi^*TM+\sO(-E)-\sO(\sD-E)$  
on $\tilde{M}$. This is an element in $K^0(\tilde{M})$ and for its Chern class we have
\begin{align*}
\pi_* \left(c(\cK)  \cap  [\tilde{M}]\right) 
=& \ \pi_*\left(  
\frac{c(\pi^*TM)(1-E)}{1+\sD-E}  \cap  [\tilde{M}] \right) \notag \\
=& \ c(TM)  \cap  [M]  -\pi_*\left( c(\pi^*TM)\cap \frac{[\sD]}{1+\sD-E} \right).
\end{align*}

Recall the following facts about Chern classes of (global) vanishing cycles.
\begin{lemm}[Theorem 3.2(i) in \cite{PP01}]\label{chernchi}
\[
\pi_*\left( c(\pi^*TM)\cap \frac{[\sD]}{1+\sD-E} \right)=c_*(\Psi_f(\ind_M)).
\]
\end{lemm}
Here to define $\Psi_f(\ind_M)$  we take a local equation $h$ for the section $f\in H^0(M,\sO(D))$ around any point $x\in D$ and consider the value $\Psi_h(\ind_{M})(x)$. This value is  independent of  the choice of the local equation $h$ and therefore they glue to a  global constructible function $\Psi_f(\ind_M)$. 

The Fulton-Johnson class $c^{FJ}(D)$ of the singular hypersurface $D$ is defined as (see \cite[Example 4.2.6]{MR1644323})
\[
c^{FJ}(D)=c(TM-\sO(D))\cap[D].
\]

\begin{lemm}\label{chernchi;isolated singularities}
If $D$ has only isolated singularities, then $c_*(\Psi_f(\ind_M))= c^{FJ}(D)$. 
\end{lemm}
\begin{proof}
The reult is well-known and can be deduced from \cite[Example 0.1]{PP01}. It is also a direct corollary of Lemma \ref{chernchi}.  Because $D$ has only isolated singularities, $\pi_*E^{i}=0$ for $i=0,\ldots, n-1$. By projection formula we have
\begin{align*}
\pi_*\left( c(\pi^*TM)\cap \frac{[\sD]}{1+\sD-E} \right)&=\pi_*\left( c(\pi^*TM)\cap \sum_{i=0}^{n-1}(E-\sD)^i\cdot [\sD] \right) \\
&=\pi_*\left( c(\pi^*TM)\cap \sum_{i=0}^{n-1}(-\sD)^i\cdot[\sD] \right) \\
&=c(TM)\cap \sum_{i=0}^{n-1}(-D)^i\cdot [D] \\
&=c(TM-\sO(D))\cap [D].
\end{align*}
\end{proof}

By Lemma \ref{chernchi}, we have
\begin{equation}\label{eq; Chernvirtual}
c_*(\Psi_f(\ind_M))=  c(TM)  \cap  [M] -\pi_* \left(c(\cK)  \cap  [\tilde{M}]\right).
\end{equation}

First we generalize \cite[Theorem 7.16]{Suwa} from isolated singularities to arbitrary singularities, assuming the foliation $\cF$  along $D$ satisfies condition $(\ast\ast)$. 
\begin{coro}
\label{coro; sumGSVonD}
Let $\cF$ be a logarithmic one-dimensional foliation    along $D$ with tangent bundle $L$ and with isolated singularities.
If $\cF$  satisfies  condition $(\ast\ast)$, then 
\[
\int_M c(TM- L)\cap [M] - \int_{\tilde{M}} c(\cK-\pi^*L)\cap [\tilde{M}] 
= \sum_{x\in \Sing(\cF)\cap D}\GSV(\cF, D, x) \/.
\]
When $D$ has only isolated singularities, the formula above is reduced to Suwa's formula 
\[
\int_D c(TM-\sO(D)-L)\cap [D] 
=  \sum_{x\in \Sing(\cF)\cap D} \GSV(\cF, D, x) \/.
\]
\end{coro}

\begin{proof}
We write $\{c_*(\Psi_f(\ind_M))\}_{k}$ for the $k$-dimensional component of $c_*(\Psi_f(\ind_M))$. Since the proper pushforward $\pi_*$  preserves dimension, by \eqref{eq; Chernvirtual} we have  
\[
\{c_*(\Psi_f(\ind_M))\}_{k} = c_{n-k}(TM)  \cap  [M] 
 -  \pi_*\left( c_{n-k}(\cK)  \cap  [\tilde{M}] \right).
\] 
By the shadow formula Proposition~\ref{prop; shadow}  we have 
\[
c_*(\Psi_f(\ind_M))
= \textup{Dual-Shadow}_{T^*M}\left(\textup{CC}(\Psi_f(\ind_M)) \right) ,
\]
or equivalently 
\[
\textup{Shadow}_{T^*M}\left(\textup{CC}(\Psi_f(\ind_M)) \right) = \sum^n_{k=0}(-1)^kc_{n-k}(TM)\cap [M] - \sum^n_{k=0}(-1)^k\pi_*\left(  c_{n-k}(\cK)  \cap  [\tilde{M}] \right).
\]
By  the   twisted shadow  formula  Proposition~\ref{prop; jettwistformula}  we then have 
\begin{align*}
&\deg \left( [\textup{JCC}(\Psi_f(\ind_M))]\cdot [M] \right) \\
=&  \sum^n_{k=0}(-1)^k\int_M c_1(L)^kc_{n-k}(TM)\cap[M] - 
\sum_{k=0}^n (-1)^{k} \int_M  c_1(L)^k\cap \pi_*\left(  c_{n-k}(\cK)  \cap  [\tilde{M}] \right) \\
=& \int_M c(TM-L)\cap [M] -  \int_{\tilde{M}}   c(\cK-\pi^*L) \cap  [\tilde{M}]  \/.
\end{align*}
On the other hand, as $\cF$  satisfies condition ($\ast\ast$), by Theorem \ref{theo; comparison} we have
\[
\deg \left( [\textup{JCC}(\Psi_f(\ind_M))]\cdot [M] \right)
= \deg \left( [\textup{JCC}(\Psi_f(\ind_M))]\cdot [\sigma_\cF(M)] \right)
= \sum_{x\in \Sing(\cF)\cap D} \GSV(\cF, D, x)  \/.  
\]
When $D$ has only isolated singularities, Lemma \ref{chernchi;isolated singularities} and equation \eqref{eq; Chernvirtual} together implies that
\[
c(TM-\sO(D))\cap[D]= c(TM)  \cap  [M] -\pi_* \left(c(\cK)  \cap  [\tilde{M}]\right)
\]
Capping both sides of the equation with $c(L)^{-1}$ we have
\[
c(TM-\sO(D)-L)\cap[D]
 =  c(TM-L)\cap [M]-\pi_*\left(c(\cK-\pi^*L)\cap [\tilde{M}] \right) \/.
\] 
This proves the second statement. 
\end{proof}

The following corollary generalizes  \cite[Theorem A]{Macha25} and  
\cite[Theorem 1-(III)]{Macha25-Mar}, which were  stated only  for   hypersurface $D$ with  isolated  singularities.
\begin{coro}
\label{coro; iso}
Let $\cF$ be a logarithmic one-dimensional foliation    along $D$ with tangent bundle $L$ and with isolated singularities.  
Assume that $\cF$  satisfies condition  $(\ast\ast)$, then we have 
\begin{align*}
&\int_{\tilde{M}}   c(\cK-\pi^*L) \cap  [\tilde{M}] \\
=&\sum_{x\in \Sing(\cF)\setminus D_{sm}} \Ind_{\PH}(\cF, x)+	
\sum_{x\in \Sing(\cF)\cap D_{sm}} \Ind_{\log}(\cF,D, x) -
\sum_{y\in \Sing(\cF)\cap \Sing(D)} \GSV(\cF, D, y) \/.
\end{align*}

If  $D$ has only isolated singularities,  we have 
\begin{align*}
 & \int_M c_n(TM-\sO(D)-L)\cap [M] 
\\
=& \sum_{x\in \Sing(\cF)\setminus D_{sm}} \Ind_{\PH}(\cF, x)+	
\sum_{x\in \Sing(\cF)\cap D_{sm}} \Ind_{\log}(\cF,D, x) -
\sum_{y\in \Sing(D)} \GSV(\cF, D, y)  \\
=&  \sum_{x\in \Sing(\cF)} \Ind_{\PH}(\cF, x)+	
\sum_{x\in \Sing(\cF)\cap D } \Sch(\cF,D, x) + (-1)^n 
\sum_{y\in \Sing(D)} \mu_D(y)   \/.
\end{align*}
\end{coro}

\begin{proof}
Combining equation \eqref{eq; sumofPH} with  Corollary~\ref{coro; sumGSVonD} we have
\begin{align*}
&\int_{\tilde{M}}   c(\cK-\pi^*L) \cap  [\tilde{M}]
= \sum_{x\in \Sing(\cF)} \Ind_{\PH}(\cF, x)-	\sum_{x\in \Sing(\cF)\cap D }  \GSV(\cF, D, x)  
\\
&= \sum_{x\in \Sing(\cF)} \Ind_{\PH}(\cF, x)-	\sum_{x\in \Sing(\cF)\cap D_{sm}}  \GSV(\cF, D, x)  -
\sum_{y\in \Sing(\cF)\cap\Sing(D)} \GSV(\cF, D, y)
 \/.	
\end{align*}

In neighborhoods of smooth points $x\in D_{sm}$ we have $\Psi_f(\ind_M)=\ind_D=\ind_{M}-\ind_{M\setminus D}$,   by Theorem \ref{theo; comparison} we have
$\GSV(\cF, D, x) = \Ind_{\PH}(\cF, x)-
\Ind_{\log}(\cF,D, x)$.
Therefore 
\begin{align*}
&\int_{\tilde{M}}   c(\cK-\pi^*L) \cap  [\tilde{M}]
 \\
=&\sum_{x\in \Sing(\cF)} \Ind_{\PH}(\cF, x)-	
\sum_{x\in \Sing(\cF)\cap D_{sm}}  \left( \Ind_{\PH}(\cF, x)-\Ind_{\log}(\cF,D, x)\right) -
\sum_{y\in \Sing(\cF)\cap \Sing(D)} \GSV(\cF, D, y) \\
=& \sum_{x\in \Sing(\cF)\setminus D_{sm}} \Ind_{\PH}(\cF, x)+	
\sum_{x\in \Sing(\cF)\cap D_{sm}} \Ind_{\log}(\cF,D, x) -
\sum_{y\in \Sing(\cF)\cap \Sing(D)} \GSV(\cF, D, y) \/.
\end{align*}

 Now we assume that $D$ has only isolated singularities. Since $s_\cF$ is tangent to $D$, by Saito's triviality lemma \cite[Proposition (3.6)]{MR586450} we have $\Sing(D)\subset\Sing(\cF)$. By projection formula, equation \eqref{eq; Chernvirtual} and Lemma \ref{chernchi;isolated singularities} we have
\begin{align*}
\pi_*\left(c(\cK-\pi^*L)\cap [\tilde{M}] \right)&=c(L)^{-1}\cap (\pi_*(\cK)\cap[\tilde{M}]) \\
&=c(L)^{-1}\cap \Big(c(TM)\cap[M]-c(TM-\sO(D))\cap[D]\Big) \\
&=c(L)^{-1}\cap \Big(c(TM-\sO(D))\cap [M]\Big) \\
&=c(TM-\sO(D)-L)\cap[M].
\end{align*} 
This proves the first equation when $D$ has only isolated singularities. The second equation is obtained from  Remark \ref{rema:GSV-SCH=mu}. 
\end{proof}

\bibliographystyle{plain}
\bibliography{ref}

\end{document}